\documentclass[aop,MSNbibl,citesort,dvips]{arximspdf}
\usepackage{mathrsfs}
\doi{10.1214/10-AOP562}
\volume{39}
\issue{2}
\pubyear{2011}
\firstpage{609}
\lastpage{647}

\makeatletter
\newtheorem{teorema}{Theorem}
\newtheorem{pro}{Proposition}

\newtheorem{cor}{Corollary}

\newproclaim{definic}{Definition}
\newproclaim{example}{Example}
\newproclaim{note}{Note}

\def\bsuffix #1{#1}

\makeatother

\begin{document}
\begin{frontmatter}

\title{Markovian bridges: Weak continuity and\\
pathwise constructions\thanksref{T1}}

\thankstext{T1}{Collaboration was possible through Grants ECOS M01 M07 and ANR-09-BLAN-0084-01 with the Universit\'e d'Angers as a host.}
\runtitle{Markovian bridges}

\begin{aug}
\author[A]{\fnms{Lo\"ic} \snm{Chaumont}\ead[label=e1]{loic.chaumont@univ-angers.fr}}
\and
\author[B]{\fnms{Ger\'onimo}
\snm{Uribe Bravo}\ead[label=e2]{geronimo@sigma.iimas.unam.mx}\thanksref{T2,T3}\corref{}}

\thankstext{T2}{Currently at the Department of Statistics, University
of California at Berkeley.}
\thankstext{T3}{Supported by CoNaCyT Grant 174498 and a postdoctoral
fellowship from UNAM, and partially conducted at UNAM's Instituto de Matem\'aticas and Paris VI's Laboratoire des
Probabilit\'es et Mod\`eles Al\'eatoires.}
\runauthor{L. Chaumont and G. Uribe Bravo}
\affiliation{Universit\'e d'Angers and Universidad Nacional Aut\'onoma
de M\'exico}
\address[A]{LAREMA\\
D\'epartement de Math\'ematiques\\
Universit\'e d'Angers\\
2, Bd Lavoisier-49045, Angers Cedex 01\\
France\\
\printead{e1}} 
\address[B]{Departamento de Probabilidad y Estad\'istica\\
Instituto de Investigaciones\\
\quad en Matem\'{a}ticas Aplicadas y en Sistemas\\
Universidad Nacional Aut\'onoma de M\'exico \\
Mexico City, A.P. 20-726\\
Mexico\\
\printead{e2}}
\end{aug}

\received{\smonth{5} \syear{2009}}
\revised{\smonth{4} \syear{2010}}

%
\begin{abstract}
A Markovian bridge is a probability measure taken from a disintegration
of the law of an initial part of the path of a Markov process given its
terminal value. As such, Markovian bridges admit a natural
parameterization in terms of the state space of the process. In the
context of Feller processes with continuous transition densities, we
construct by weak convergence considerations the only versions of
Markovian bridges which are weakly continuous with respect to their
parameter. We use this weakly continuous construction to provide an
extension of the strong Markov property in which the flow of time is
reversed. In the context of self-similar Feller process, the last
result is shown to be useful in the construction of Markovian bridges
out of the trajectories of the original process.
\end{abstract}




\end{frontmatter}
%
\section{Introduction and main results}
\subsection{Motivation}
The aim of this article is to study Markov processes on $[0,t]$,
starting at $x$, conditioned to arrive at $y$ at time $t$.
Historically, the first example of such a conditional law is given by
Paul L\'evy's construction of the Brownian bridge: given a Brownian
motion $B$ starting at zero, let
\[
b^{x,y,t}_s=x+B_s-B_t\frac{s}{t}+( y-x)\frac{s}{t}.
\]
Then $b^{x,y,t}$ is a version of $B$ started at $x$ and conditioned on
$B_t=y$ in the sense that
\[
\mathbb{E}\bigl( F\bigl( ( x+B_s)_{s\in[0,t]}\bigr)f( x+B_t) \bigr)=\int
\mathbb{E}( F( b^{x,y,t}) )f( y)\mathbb{P}( x+B_t\in dy ).
\]
This example synthesizes the main considerations of this work: one is
able to construct a specific version of the disintegration of the law
of $( B_s)_{s\in[0,t]}$ given $B_t$ which is weakly continuous,
and one is able to give a pathwise construction of this conditional law
out of the trajectories of $B$. Since there is at most one weakly
continuous disintegration, it is natural to look for conditions
guaranteeing its existence and to characterize it, which we do in the
class of Feller processes. Kallenberg has given in \cite{backwardkallenberg} a very general result for the existence of weakly
continuous bridges of L\'evy processes using the convergence criteria
for processes with exchangeable increments of \cite{kallenbergExchangeable}. It is a consequence of our results. The first
abstract framework for the existence of particular bridge laws in the
context of Markov process is \cite{markovbridges}. It is different
from the one adopted in this work since they rely on duality
considerations while we rely mainly on the Feller property. A more
recent study of the question of existence of bridge laws without
duality hypotheses is \cite{BarczyPap}. The framework is similar to
ours although the laws are constructed on a product space by
Kolmogorov's extension theorem; we supplement their construction by
providing an analysis of path regularity of bridges and, in doing so,
end up with a different proof of existence which relies mainly on weak
convergence.

Paul L\'evy's Gaussian construction of the Brownian bridge is of
limited applicability in the context of Markov processes. However, he
also gave a different pathwise construction of a Brownian bridge with a
Markovian flavor: let $g$ be the last zero of $B$ before time $1$,
which is not zero since $B$ comes back to zero at arbitrarily small
times, and set
\[
b_t=\frac{1}{\sqrt{g}}B_{g\cdot t}
\]
for $t\leq1$. Then $b$ and $b^{0,0,1}$ have the same law. We will
provide further examples of this type of pathwise construction, which
in the case of Brownian motion is given as follows. Let $g_c=\sup\{
t\leq1\dvtx B_t=c\sqrt{t}\}$; the positivity of $g_c$ for any $c\in
\mathbb{R}$
is an immediate consequence of the asymptotic behavior of the Brownian
curve at~$0$. If
\[
b^c_t=\frac{1}{\sqrt{g_c}}B_{g_c\cdot t}
\]
for $t\leq1$, then $b^c$ and $b^{0,c,1}$ have the same law. To compute
the law of $b^c$, we extend the usual strong Markov property: note that
$\{ g_c>t\}\in\sigma( B_u\dvtx u\geq t)$ so that $g_c$ is a kind of
backward optional time at which a version of the strong Markov property
holds, the law of $( B_{s\wedge g_c})_{t\geq0}$ given $\sigma(
B_u\dvtx u\geq g_c)$ is that of a Brownian bridge from $0$ to $c\sqrt
{g_c}=B_{g_c}$ of length $g_c$. Applying Brownian scaling gives the
desired result. Looking at our results in the preliminary draft, Marc
Yor noticed and pointed out to us the following short proof in the
special case of Brownian motion: by time-inversion, $\tilde B_t=t
B_{1/t}$ is a Brownian motion and
\[
g_c=1/\inf\bigl\{ t\geq1\dvtx \tilde B_{t}=c\sqrt{t}\bigr\};
\]
denote by $T$ the stopping time appearing in the denominator. By the
strong Markov property and scaling, $X_t=\tilde B_{T( 1+t)}/\sqrt
{T}-c$, $t\geq0$, is another Brownian motion, so that
\[
t\bigl( X_{( 1-t)/t}-c\bigr)=B_{tg_c}/\sqrt{g_c}-t c,\qquad t\in[0,1],
\]
has the same law as $b^{0,0,1}$ and so $b^c$ and $b^{0,c,1}$ have the
same law.

Note, however, that our methods will apply to self-similar processes
which do not posses the time inversion property. In particular, we will
study the case of stable L\'evy processes.


\subsection[The results]{Statement of the results}\label{shiftOperatorsDef}
We will work on an arbitrary locally compact metric space with a
countable base (or LCCB for short) denoted $( S,\rho)$. On it we will
consider the Borel $\sigma$-field denoted $\mathscr{B}_{S}$ and
$b\mathscr{B}_{S}$ will
stand for the
set of measurable and bounded functions from $S$ to $\mathbb{R}$. We will
consider a \textit{Markovian family} of probability measures on this
space which satisfy the Feller property, by which the following is
meant. Let $D_\infty$ ($D_t$) stand for the Skorohod space of
c\`adl\`ag functions from $[0,\infty)$ ($[0,t]$) into $S$ and consider
on it the shift operators $\theta_t\dvtx D_\infty\to D_\infty$
given by $\theta_t f\dvtx s\to f( t+s)$
(they can also be defined on $[0,t']$ if $t'>t$). Let
$X=( X_s)_{s\geq0}$ denote the canonical process, and write $\mathscr{F}$
and $( \mathscr{F}_s)_{s\geq0}$ for the $\sigma$-field and the
canonical filtration
generated by $X$.

\begin{definic*}
A \textit{Markovian family} on $( S,\rho)$ is a collection of probability
measures $( \mathbb{P}_{x})_{x\in S}$ on $D_\infty$
indexed by
the elements of $S$ which satisfies

\textit{Starting point property}: For all $x\in S$,
\[
\mathbb{P}_x( X_0=x)=1.
\]

\textit{Measurability property}: For all $F\in b\mathscr{F}$,
\[
x\mapsto\mathbb{E}_{x}( F)
\]
is measurable.

\textit{Markov property}: For every $F\in b\mathscr{F}_s$ and every
$G\in b\mathscr{F}$,
\[
\mathbb{E}_{x}( F\cdot G\circ\theta_s)=\mathbb{E}_{x}\bigl( F\cdot\mathbb{E}_{X_s}( G)\bigr).
\]

A Markovian family $( \mathbb{P}_x)_{x\in S}$ is said to satisfy
the Feller property (and we will therefore speak of a \textit{Feller
family}) if the operators $( P_s)_{s\geq0}$ defined on $b\mathscr{B}_{S}$
by means of $P_sf( x)=\mathbb{E}_x( f( X_s))$ map the
space of continuous functions $f\dvtx S\to\mathbb{R}$ which vanish at
infinity into itself and for any such $f$ we have:
\[
\lim_{t\downarrow0}\Vert P_sf-f\Vert=0,
\]
where $\Vert\cdot\Vert$ denotes the uniform norm.
\end{definic*}

Of course, Feller families are in bijection with (conservative) Feller
semigroups. In this case, we even have the strong Markov property at
every stopping time $T$: for every $F\in b\mathscr{F}_T$ and every
$G\in b\mathscr{F}$,
\[
\mathbb{E}_{x}( F\cdot G\circ\theta_T)=\mathbb{E}_{x}\bigl( F\cdot\mathbb{E}_{X_T}( G)\bigr).
\]

We seek to build a version of the conditional law of
$( X_s)_{s\leq t}$ given $X_t=y$ under $\mathbb{P}_x$, which we would call
Markovian bridge from $x$ to $y$ of length $t$. One could appeal to the
general theorem on existence of regular conditional distributions (see,
e.g., \cite{kallenberg}, Theorem 6.3, page 107), but that result
builds the whole family of conditional laws as $y$ varies and does not
give control over individual conditional laws. Since we are working on
a Polish space, we might impose further regularity conditions on
conditional laws such as their weak continuity as $y$ varies; since
there is at most one weakly continuous disintegration with respect to
the extremal values, this singles out specific conditional laws. This
is the strategy we will follow. To that end, consider a Feller family
$( \mathbb{P}_{x})_{x\in S}$ on $( S,\rho)$ and its associated semigroup
$P=( P_s)_{s\geq0}$ and suppose that $P_s$ admits a transition
density $p_s( \cdot,\cdot)$ with respect to a $\sigma$-finite
measure $\mu$ on $( S,\rho)$ in the sense that
\[
P_s f( x)=\int f( y)p_s( x,y)\mu ( dy).
\]


Fix $x\in S$ and set $\mathscr{P}_{t}=\{ y\dvtx p_t( x,y)>0\}$. Under the
hypotheses
\begin{longlist}
\item[(H1)] $y\mapsto p_s( x,y)$ is continuous for all $s\in(0,t]$,
\item[(H2)] the Chapman--Kolmogorov equations
\[
p_t( x,y)=\int p_{t-s}( x,z)p_s( z,y)\mu( dz)
\]
hold for each $y\in\mathscr{P}_{t}$, and for $0<s<t$, and
\item[(H3)] $s\mapsto p_s( x,y)$ is continuous for
all $x,y\in S$,
\end{longlist}
which are more clearly explained in Section \ref{construction}, we
prove our basic existence result.

\begin{teorema}\label{bridgeByWeakCont}
For every $y\in S$ such that $p_t( x,y)>0$, the laws
\[
\mathbb{P}_x\bigl(\cdot\vert X_t\in B_{\delta}( y)\bigr)
\]
converge weakly as $\delta\to0$ to a law $\mathbb{P}_{x,y}^t$ such that:
\begin{longlist}
\item[\textup{(1)}] $y\mapsto\mathbb{P}_{x,y}^t$ is weakly continuous, and
\item[\textup{(2)}] for every $f\in b\mathscr{B}_{S}$ and $F\in b\mathscr{F}_t$,
\[
\mathbb{E}_x\bigl( F\cdot f( X_t)\bigr)=\int_{\{ y:p_t( x,y)>0\}}\mathbb{E}_{x,y}^t( F)p_t( x,y)\mu( dy).
\]
\end{longlist}
\end{teorema}

\begin{note*}
Given $x\in S$, $t>0$, $s\in(0,t)$ and $y$ such that $p_t( x,y)>0$,
the Chapman--Kolmogorov equations of (H2) hold as
consquence of the continuity assumption (H1) if additionally $p_{t-s}(
\cdot, y)$ is bounded. 
\end{note*}

Special cases of Theorem \ref{bridgeByWeakCont} are found in the
literature: Kallenberg proves the weak continuity and the approximation
for a subclass of L\'evy processes in \cite{backwardkallenberg}, the
special case of stable L\'evy processes (when the starting and ending
points are zero) is obtained by Bertoin in \cite{bertoinLevyProcesses}, VIII.3, Proposition
11, by scaling arguments and using excursion
theory by Chaumont in \cite{ChaumontThesis,ChaumontBSM}.

By the same method of proof, we can study joint weak continuity in the
starting and ending point and the length. However, since bridge laws
associated to different lengths are defined on different Skorohod
spaces, we need to specify the interpretation of weak continuity we
will use. For every $f\in D_t$, we can associate the function
$f^t\in D_\infty$ given by $f^t( s)=f( s\wedge t)$.
This measurable mapping will be denoted by $i_t$ and 
we will say that the sequence of measures $\mathbb{P}_n^{t_n}$ on $D_{t_n}$
converge weakly if $\mathbb{P}_n^{t_n}\circ i_{t_n}^{-1}$ converges
weakly in
$D_\infty$. To simplify notation, from this point on, we will think
of bridge measures as defined on $D_\infty$ by identifying
$\mathbb{P}_{x,y}^t$ with $\mathbb{P}_{x,y}^t\circ i_t^{-1}$. Kallenberg
used in \cite{backwardkallenberg} another notion of weak continuity
with respect to the temporal parameter; it differs only when one
considers lengths that go to infinity.

A technical hypothesis, related to the joint weak continuity of bridge
laws with respect to the ending point and the length, is the
following:
\begin{longlist}
\item[(H1$^{\prime}$)]
$( s,y)\mapsto p_s( x,y)$ is continuous for all $x\in S$.
\end{longlist}
Another one, related to weak continuity with respect to all variables is
\begin{longlist}
\item[(H1$^{\prime\prime}$)] $( s,y,x)\mapsto p_s( x,y)$ is
continuous.
\end{longlist}

We have the following corollary.
\begin{cor}\label{jointWeakContinuityBridgeLaws}
Under \textup{(H1$^{\prime}$)} and \textup{(H2)}: the
bridge laws $( \mathbb{P}_{x,y}^t)$ are jointly continuous in
$y$ and $t$. Under \textup{(H1$^{\prime\prime}$)} and \textup{(H2)}, the bridge laws are
weakly continuous with respect to $x,y$ and $t$.
\end{cor}

Now we will analyze a generalization of the usual strong Markov
property for Feller processes in which Markovian bridges play a
prominent role.

Let us define, for a fixed time~$t$, the $\sigma$-fields associated to
the past
before time~$t$, $\mathscr{F}_t$, and to the future after time $t$,
$\mathscr{F}^t=\sigma( X_s\dvtx s\geq t)$
and place ourselves under (H1)--(H3); thanks to the Markov
property, we obtain the following.

The conditional law of
$X^{s,t}=( X_{( r+s)\wedge t})_{r\geq0}$ given $X_s,X_t$
under $\mathbb{P}_x$ is $\mathbb{P}_{X_s,X_t}^{t-s}$.

We shall generalize the preceding conditional description to a
\textit{strong Markov property with respect to future events}.
Actually, the method of proof will be analogous to a known one for the
strong Markov property: we will discretize the problem, then we shall
use the local property of conditional expectation (to be stated
shortly), and finally, continuity considerations will be used to
transport conclusions of the discrete setup to the continuous one. The
target result needs the following.

\begin{definic*}
A \textit{backward optional time} is a random variable
$L\dvtx D_\infty\to[0,\infty]$ such that $\{ L>t\}\in\mathscr{F}^t$
for all
$t>0$.

For a backward optional time $L$, the \textit{$\sigma$-field of
events occurring after $L$}, denoted $\mathscr{F}^L$, is defined to be
$\sigma( X\circ\theta_L)$.
\end{definic*}

As a first example, let us note that if $U\subset S$ is open, then
the \textit{last visit to $U$} equal to zero if $X$ is never in $U$ and
equal to
\[
\sup\{ s\geq0\dvtx X_s\in U\}
\]
otherwise is a backward optional time. A second example would be the
last visit to an open set (just) before a fixed time $t$ given by
\[
L_U^t=
\cases{
0,&\quad if $\{ s<t\dvtx X_s\in U\}=\varnothing$,\cr
\sup\{ s<t\dvtx X_s\in U\},&\quad otherwise.}
\]
The first example belongs to the following class of random times, which
are all backward optional times.

\begin{definic*}
A \textit{cooptional time} is a random variable
$L\dvtx D_\infty\to[0,\infty]$ such that
$L\circ\theta_t=( L-t)^+$.
\end{definic*}

Cooptional times are backward optional times since, by definition they
are random variables, and then
\[
\{ L>t\}=\{ ( L-t)^+>0\}=\theta_t^{-1}( \{ L>0\})\in\mathscr{F}^t.
\]
However, the last visit to an open set before a fixed time is an
example of a backward optional time which is not cooptional.

Backward optional times are the key to opening random temporal windows
in the Markov property. However, to provide a statement closer to the
usual expression of the strong Markov property, we will use the
\textit{shift and stop operators}
$\sigma^s_t\dvtx D_\infty\to D_\infty$ given by
\[
\sigma^s_t f( r)=
\cases{
f( r+s),&\quad if $r+s<t$,\cr
f( t-),&\quad if $r+s\geq t$.}
\]
Since these operators were defined in terms of $f( t-)$ instead
of $f( t)$, they are continuous on $D_\infty$ (or on $D_{t'}$ if $t'>t$).

To make sense of the following result, let us recall that we have
identified bridge laws on $D_t$ with their image on $D_\infty$
under the embedding $i_t\dvtx( f( s))_{s\in [0,t]}\mapsto( f( s\wedge
t))_{s\geq0}$.

\begin{teorema}[(The backward strong Markov property)]\label{backwardSMP}
Under \textup{(H1$^{\prime}$)} and \textup{(H2)}, the mapping
$( t,x,y)\mapsto\mathbb{E}_{x,y}^t( F)$ is measurable
for any measurable $F\dvtx D_\infty\to\mathbb{R}$. Let $S$ and $L$ be a
stopping and a backward time respectively. Then for any initial
distribution $\nu$ on $S$ and any $F\in b\mathscr{F}$,
\[
\mathbb{E}_\nu( F\circ\sigma^S_L\vert \mathscr{F}_S,\mathscr
{F}^L,X_{L-})=\mathbb{E}_{X_S,X_{L-}}^{L-S}( F)
\]
almost surely on $\{ S<L<\infty\}$.
\end{teorema}

The last theorem simply says that the process between a stopping time
$S$ and a backward optional time $L$ is a Markov bridge of random
length $L-S$ between its starting point $X_S$ and its ending point
$X_{L-}$. It also implies that $\sigma^S_L$ and $\mathscr{F}_S\wedge
\mathscr{F}^L$
are conditionally independent given $X_S$, $X_{L-}$ and $L-S$. This
result was stated by Kallenberg for L\'evy processes in \cite{backwardkallenberg} and, in a different framework, by Fitzsimmons,
Pitman and Yor in \cite{markovbridges}. Our point of view is that, as
for the usual strong Markov property for Feller processes, it is
trivially true in discrete time and that to pass to continuous time,
continuity considerations are useful. We can find many examples of
generalizations of the strong Markov property to random times; such
generalizations consist of two parts: a
statement of conditional independence of past and future with respect
to the present at a given random time $\tau$, and a description of the
conditional law of the pre-$\tau$ and post-$\tau$ parts of the process
given some notion of the present, which can be the $\sigma$-field
generated by
$\tau$ and $X_\tau$, or only $X_\tau$, or even more exotic ones.
See, for example, \cite{jacobsenPitmanConditioning,getoorSharpeCooptional,getoorSharpeMarkovProperties} for examples of
conditional independence (and several notions of present) and
\cite{meyerSmytheWalsh,jacobsen,millarMinimum} for examples where the
post-$\tau$ process is also analyzed.

We now turn to a pathwise construction of bridges of self-similar
Feller processes. We will focus on the state space $S=[0,\infty)$ or
$\mathbb{R}$, which contains~$0$.

\begin{definic*}
The \textit{scaling operators} $S^\gamma_v\dvtx D_\infty\to D_\infty$
are defined by
\[
S^\gamma_v f( t)=v^{1/\gamma}f( t/v).
\]
A Feller family $( \mathbb{P}_x)_{x\in S}$ is said to be
\textit{self-similar of index $\gamma$} if for every $x\in S$ and every
$v>0$, the image of $\mathbb{P}_x$ under the scaling operator
$S^\gamma_v$ is
$\mathbb{P}_{v^{1/\gamma}x}$.
\end{definic*}

We now give a pathwise construction of bridge laws associated to a
self-similar Feller family $( \mathbb{P}_x)_{x\in S}$ of index $\gamma$,
the bridges going from $0$ to any element of $S$ and of any length.
Suppose that $( \mathbb{P}_x)_{x\in S}$ satisfies (H1$^{\prime}$) and (H2);
explicit examples will be given in Section \ref{pathwise}. The
hypotheses ensure the applicability of Theorems \ref{bridgeByWeakCont}
and~\ref{backwardSMP}. Also, note that the image of $\mathbb
{P}_{x,y}^t$ under
the scaling operator $S^\gamma_v$ is
$\mathbb{P}_{v^{1/\gamma}x,v^{1/\gamma}y}^{t v}$; this can be
verified by the
approximation of bridge laws (Theorem \ref{bridgeByWeakCont}) and the
self-similarity property of $( \mathbb{P}_x)_{x\in S}$ using the
continuity of the scaling operators on Skorohod space.

For $c\in S$, define the random set
\[
\mathscr{Z}_c=\{ t\leq1\dvtx X_{t-}=ct^{1/\gamma}\}
\]
as well as the random time $g_c\dvtx D_\infty\to[0,\infty)$
\[
g_c=
\cases{
0,&\quad if $\mathscr{Z}_c=\varnothing$,\cr
\sup\mathscr{Z}_c,&\quad otherwise.}
\]
%

\begin{teorema}\label{pathwiseTheorem}
If $g_c>0$ $\mathbb{P}_0$-almost surely, the
law of
$( Y_t)_{t\in[0,1]}$ given by
\[
Y_t=
\cases{\displaystyle
\frac{1}{g_c^{1/\gamma}}X_{s\cdot g_c},&\quad if $t<1$,\cr
c,&\quad if $t\geq1$,}
\]
under $\mathbb{P}_0$ is $\mathbb{P}_{0,c}^1$.
\end{teorema}

Note that by the scaling relationship of bridge laws, we get the
following corollary under the hypotheses of Theorem \ref{pathwiseTheorem}: let $t>0$ and $x\in S$ be given and define
$c=xt^{-1/\gamma}$, then
the law of $( Y^x_s)_{s\in[0,t]}$ given by
\[
Y^x_s=
\cases{\displaystyle
\frac{t^{1/\gamma}}{g_c^\gamma}X_{s\cdot( g_c)/t},&\quad if $s<t$,\cr
x,&\quad if $s\geq t$,}
\]
under $\mathbb{P}_0$ is $\mathbb{P}_{0,x}^t$. It is therefore
important to provide
examples where $g_c>0$ $\mathbb{P}_0$-almost surely; the reader should be
warned by the following one: if $( \mathbb{P}_x)_{x\in\mathbb{R}}$
is the
Feller family associated to a stable L\'evy process of index $\alpha
\in(0,1)$ which has jumps of both signs, then $g_0=0$ almost surely,
because points are polar for them (cf. \cite{bertoinLevyProcesses}, II.5), while $g_c>0$ almost surely for all
$c\neq0$, as will be proved in Section~\ref{pathwise}. For symmetric
stable L\'evy processes of index $\alpha\in(0,1)$, this had been
proved in \cite{jakubowski}, Corollary 14.

When $\mathbb{P}_x$ is the law of linear Brownian motion started at
$x$, we
have computed the moments of $g_c$ in order to compare its law with
the Beta type (recall that $g_0$ has a Beta law thanks to P. L\'evy's
first arcsine law; cf. Exercise III.3.20 of \cite{revuzYor}). For this,
we define the function
\[
H_q( x)=\int_0^\infty e^{-xz-z^2/2}z^{q-1} \,dz.
\]
This function can be expressed in terms of the Hermite functions of
negative index; see \cite{lebedev65}, Section 10.2-5, especially
formula (10.5.2).

\begin{pro}\label{HermiteProposition}
We have
\[
\mathbb{E}( g_c^q )=\frac{\Gamma( 2q)}{2qH_{2q}( c)H_{2q}( -c)}\quad
\mbox{and}\quad\mathbb{E}( g_c^q )\sim\frac{e^{-c^2/2}}{\sqrt{\pi q}}
\]
as $q\to\infty$.
\end{pro}

Note that the asymptotic behavior of the moments of $g_c$ is only
compatible with a Beta law whose second parameter is $1/2$. The
explicit computation $H_q( 0)=2^{q/2-1}\Gamma( q/2)$
reproduces Paul L\'evy's arcsine law with help of the duplication
formula for the $\Gamma$ function.

Our next application of Theorems \ref{bridgeByWeakCont} and
\ref{backwardSMP} is related to stable subordinators and is obtained by
a Doob transformation. Let $\mathbb{P}_x^\alpha$ the law of a stable
subordinator of index $\alpha\in(0,1)$ starting at $x$. As Section
\ref{levyProcExample} shows us, $( \mathbb{P}_x^\alpha)_{x\geq0}$
is a
self-similar Feller family for which hypotheses (H1$^{\prime}$) and (H3) hold
(taking $\mu$ equal to Lebesgue measure). The transition density
$p_t^\alpha$ can be expressed in terms of the density $f^\alpha_t$ of
$X_t$ under $\mathbb{P}_0^\alpha$ as follows:
\[
p_t^\alpha( x,y)=f_t^\alpha( y-x).
\]
It is possible to compute the potential density $u^\alpha$ given by
\[
u_\alpha( a)=\int_0^\infty f^\alpha_t( a)\,
dt=\frac{1}{C\Gamma( \alpha)a^{1-\alpha}}\mathbf{1}_{a>0},
\]
as shown in \cite{satopl}, Example 37.19, page 261.

For any $0<b$, we define
\[
h_\alpha\dvtx[0,\infty)\to[0,\infty)=
\cases{
u_\alpha( b-a),&\quad if $a\leq b$,\cr
0,&\quad otherwise.}
\]
With it, we will consider the Doob $h_\alpha$-transform of $\mathbb{P}
_a^\alpha$, denoted $\mathbb{P}^{h_\alpha}_a$; it is a measure on the
Skorohod space $D_\infty$ on $[0,\infty)\cup\{ \Delta\}$
($\Delta$ is an additional isolated point called the cemetery)
concentrated on trajectories with values on $[0,b)\cup\{ \Delta\}$.
It is the (only) probability measure such that for all $A\in\mathscr{F}_s$
\[
\mathbb{P}^{h_\alpha}_a( A\cap\{ s<\zeta\})=\mathbb{E}^{\alpha
}_a\biggl( \frac{h_\alpha( X_s)}{h_\alpha( x)}\mathbf{1}_{A}\biggr).
\]
The family $\mathbb{P}^{h_\alpha}_a$, $a\in[0,b)\cup\{ \Delta\}$, is
Markovian and is associated to the Markov process termed the \textit{stable subordinator conditioned to die at $b$}. The terminology is
justified since $\Delta$ is absorbing and if the death-time $\zeta$
is defined as $\inf\{ t\dvtx X_t=\Delta\}$, then $X_{\zeta-}=b$ $\mathbb{P}
_a$-almost surely for every $a<b$ (cf. \cite{chaumontpathdecomp}). Our
next result is a pathwise construction of the conditioned stable
subordinator in terms of the subordinator itself: let $L=\sup\{ t\geq
0\dvtx X_t<b\}$ (which is finite under $\mathbb{P}_0^\alpha$), $g=X_{L-}$
and define $Y=( Y_t)_{t\geq0}$ as follows:
\[
Y_t=
\cases{
\displaystyle\frac{b}{g}X_{t( g/b)^\alpha},&\quad if $\displaystyle t( g/b)^\alpha
<L$,\cr
\Delta,&\quad otherwise.}
\]

\begin{teorema}\label{conditionedSubordinatorTheorem}
The law of $Y$ under
$\mathbb{P}_0^\alpha$ is $\mathbb{P}^{h_\alpha}_0$.
\end{teorema}

The paper is organized as follows: we give examples of weakly
continuous Markov bridges in Section \ref{examples}, we then prove
Theorem \ref{bridgeByWeakCont} and Corollary
\ref{jointWeakContinuityBridgeLaws} regarding construction and weak
continuity of Markovian bridges in Section \ref{construction}, passing
to the backward strong Markov property (Theorem \ref{backwardSMP}) in
Section \ref{backward}. The pathwise constructions of Markovian bridges
for self-similar processes of Theorem \ref{pathwiseTheorem} as well as
the additional computations for the Brownian case of Proposition
\ref{HermiteProposition} are given in Section \ref{pathwise}, which
contains also the construction of the conditioned stable process of
Theorem \ref{conditionedSubordinatorTheorem}.

\section{Examples of weakly continuous Markovian bridges}\label{examples}
In this section, we will give some examples of Feller
processes for which bridges can be built using Theorem
\ref{bridgeByWeakCont}; the emphasis is on checking the hypotheses
enabling us to use it.

We will start with a description of the probabilistic objects to
consider: Brownian motion and other L\'evy processes, and Bessel
processes. With this, we will introduce the associated bridges. At some
points, we will need facts concerning L\'evy processes and Bessel
processes. Although more precise information will follow, our main
references will be \cite{bertoinLevyProcesses,satopl} and
\cite{revuzYor}. In the examples that follow, the LCCB space $( S,\rho)$
will be either $\mathbb{R}$, $\mathbb{R}^n$ or $\mathbb
{R}_+=[0,\infty)$, endowed with
the usual metrics.

\subsection{Bridges of L\'evy processes}\label{levyProcExample}
We will now construct bridges of L\'evy process
and reproduce, from the point of view of our theory, the weakly
continuous construction of L\'evy bridges of \cite{backwardkallenberg};
the unproved facts can be consulted in \cite{bertoinLevyProcesses} or
\cite{satopl}.

Consider a L\'evy process $\xi$ (that is a c\`adl\`ag process starting at
zero with stationary and independent increments) and denote its law by
$\mathbb{P}$. $\xi$ is characterized by its \textit{characteristic exponent}
$\Psi$ which satisfies
\[ 
\mathbb{E}( e^{i u\xi_t})=e^{-t\Psi( u)}.
\]
If the trajectories of $\xi$ are increasing, so that it is a
subordinator, one can instead use its \textit{Laplace exponent} $\Phi$
given by
\[
\mathbb{E}( e^{-q\xi_t} )=e^{-t\Phi( q)}.
\]

If $\mathbb{P}_x^\Psi$ denotes the law of $\xi
+x$, then
$( \mathbb{P}_x^\Psi)_{x\in\mathbb{R}}$ is a Feller Markov family;
in the
case of
subordinators, $( \mathbb{P}^\Psi_x)_{x\geq0}$ is also Feller. Suppose
now that $\Psi$ is such that $\exp( -t\Psi)$ is integrable for any
$t>0$ (this corresponds to hypothesis (C) in
\cite{backwardkallenberg}); by Fourier inversion, one can prove that
the law of $\xi_t$ is absolutely continuous and admits a jointly
continuous density $f^\Psi_t$ bounded on
$[t,\infty)\times\mathbb{R}$ (the second factor is $\mathbb{R}_+$
in the subordinator
case) for each $t>0$. By independence and homogeneity of the
increments, the transition density $p_t^\Psi$ for $X_t$ under $\mathbb{P}_x$
can be taken equal to $f^\Psi_t( \cdot-x)$, which implies the
validity of hypotheses (H1$'$) and (H2), where the latter holds by
the bounded character of the density.

In \cite{sharpePositivityOfDensity}, it is proven that $f^\Psi_t$ is
positive on the interior of the support of the law of $\xi_t$, which is
of the form $(dt,\infty)$ for all $t>0$ or $(-\infty,dt)$ for all
$t>0$, where $d\in[-\infty,\infty]$; $|d|=\infty$ if the absolute
value of the L\'evy process is not a subordinator and it is finite
otherwise.

In particular, we can apply the preceding reasoning to \textit{stable
L\'evy processes of index $\alpha\in(0,2]$} since the characteristic
exponent satisfies
\[
\bigl|e^{-t\Psi( u)}\bigr|=e^{-tC|u|^{\alpha}}
\]
for some $C>0$. Stable L\'evy processes are the only L\'evy processes
whose Markovian family is self-similar. This includes Brownian motion,
whose characteristic exponent is $\Psi( u)=u^2/2$ and the
corresponding transition density is given explicitly as
\[
p_t( x,y)=\frac{1}{\sqrt{2\pi t}}e^{-( y-x)^2/2t}.
\]
We remark that L\'evy's representation gives a simpler way of deducing
the existence and weak continuity of Brownian bridges both in the
one-dimensional and multi-dimensional cases.

\subsection{Bridges of Bessel processes}\label{besselProcessExample}
The next family of processes we shall
consider is that of Bessel processes of dimension $\delta\in
[0,\infty)$. When $\delta\in\mathbb{Z}_+$, the law of the \textit{Bessel process} of dimension $\delta\in\mathbb{Z}_+$ starting at
$x$, denoted
$\mathbb{P}^\delta_x$, is the law of $\|\vec{x}+\vec{B}\|$ where
$B$ is a
$\delta$-dimensional Brownian motion for any vector $\vec{x}$ such that
$\|\vec{x}\|=x$. In \cite{revuzYor}, VI.3.1, page 251, it is argued that
$( \mathbb{P}^\delta_{x})_{x\in[0,\infty)}$ is a Markovian family; its
Feller property is immediate from that of Brownian motion. The case
$\delta\notin\mathbb{Z}_+$ is handled via stochastic differential equations
in \cite{revuzYor}, XI.1; its law will be denoted by
$\mathbb{P}_x^\delta$ and it constitutes
a Feller
family on $[0,\infty)$ whose transition density with respect to
Lebesgue measure, which is expressed in a simpler fashion in terms of
the \textit{index} $\nu=\delta/2-1$ associated to the dimension
$\delta>0$ and the \textit{modified Bessel function of the first kind},
denoted $I_\nu$, given by
%
\begin{equation} \label{inu}
I_\nu( x)=\sum_{k=0}^\infty
\biggl( \frac{x}{2}\biggr)^{\nu+2k}\frac{1}{k!\Gamma( 1+\nu+k)}.
\end{equation}
The transition density is given by
\[
p_t^\delta( x,y)=\frac{1}{t}\biggl( \frac{y}{x}\biggr)^{\nu
}ye^{-( x^2+y^2)/2t}I_\nu\biggl( \frac{xy}{t}\biggr)
\]
for $x>0$ and $t>0$. For $x=0$, we have the expression
\[
p_t^\delta( 0,y)=\frac{y^{2\nu+1}}{2^\nu
t^{\nu+1}\Gamma( \nu+1)}e^{-y^2/2t}.
\]
This transition density satisfies hypotheses (H1$'$) and (H2).
This is because we can bound the transition density using the
asymptotic equality
\[
I_\nu( x)\sim\frac{1}{\sqrt{2\pi x}}e^x
\]
valid as $x\to\infty$ (cf. \cite{lebedev65}, 5.11.10, page 123), which implies
\[
\sup_{x\in\mathbb{R}_+,y\leq M}p_t^\delta( x,y )<\infty
\]
for any $M>0$. We can therefore construct Bessel bridges from $x$ to
$y$ for any $x\geq0$ and $y>0$. It is possible to consider $y=0$ for a
bridge law if instead of using Lebesgue measure $\lambda$, we use the
$\sigma$-finite measure with density $y\mapsto y^{2\nu+1}$ with
respect to Lebesgue measure, which would imply the fact that the
transition density of $\{ \mathbb{P}_x^\delta\dvtx x\in[0,\infty)\}$ with
respect to it assigns a positive value to $0$ starting from any $x\in
[0,\infty)$, and satisfies hypotheses (H1$'$) and (H2).

\subsection{Bridges of Bessel processes with drift}\label{besselProcessWithDriftExample}
Bessel processes are particular
instances of Bessel processes in the wide sense, introduced in
\cite{watanabeWideSenseBessel}, which will provide the next example of
stochastic processes for which one can build bridges by weak
continuity. Let $\delta>0$, $c\geq0$ and consider $\nu=\delta/2-1$
and
\[
\rho_c( x)=2^{\nu}\Gamma( 1+\nu)\bigl( \sqrt{2c} x\bigr)^{-\nu}I_\nu\bigl(
\sqrt{2c} x\bigr),
\]
where $I_\nu$ is the modified Bessel function of the first kind given
in \ref{inu}. A \textit{Bessel process in the wide sense} with index
$( \delta,c)$ is a diffusion process on $[0,\infty)$ determined
by the local generator
\[
L^{\delta,c}=\frac{1}{2}\frac{\partial}{\partial
x^2}+\biggl( \frac{\delta-1}{2x}+\frac{\rho'_c( x)}{\rho_c( x)}\biggr)\frac
{\partial}{\partial
x};
\]
the point 0 is a reflecting boundary when $0<\alpha<2$ and an entrance
boundary for $\alpha\geq2$. When $c=0$, this is just an ordinary
Bessel process. Their law starting at $x$ will be denoted $\mathbb{P}
_x^{\delta,c}$. Bessel processes in the wide sense can also be
interpreted as Bessel processes with drift: for integer $\delta\geq
1$, $\mathbb{P}_0^{\delta,c}$ is the law of the modulus of $\delta
$-dimensional Brownian motion with a drift vector $\vec{c}$ of length
$c$ that starts at zero (cf.~\cite{pitmanYorBesselWithDrift}, Remark 5.4.iii, page
319). The last result is actually proved
through a third description of Bessel processes in the wide sense
contained in \cite{pitmanYorBesselWithDrift}, Sections 3 and~4: the
law $\mathbb{P}_x^{\delta,c}$ is locally absolutely continuous with respect
to $\mathbb{P}_x^\delta$. To describe this relationship, we introduce
$\alpha
=\sqrt{2c}$, the hitting-time $T_y$ of $y$ by the canonical process,
and the functions
\[
\phi_\alpha( x,y)=\mathbb{E}_x^\delta( e^{-\alpha T_y})
\quad\mbox{and}\quad\phi_{\alpha\uparrow}( y)=
\cases{
\phi_\alpha( x_0,y),&\quad$y\leq x_0$,\cr
1/\phi_\alpha( x_0,y),&\quad$y> x_0$,}
\]
where $x_0$ is any element of $(0,\infty)$; the choice of $x_0$
affects the definition of $\phi_{\alpha\uparrow}$ by a constant
factor, as can be seen from \cite{itomckean}. For any $t$, the
restriction of $\mathbb{P}_x^{\delta,c}$ to $\mathscr{F}_t$ is
absolutely continuous
with respect to the restriction of $\mathbb{P}_x^\delta$ to $\mathscr
{F}_t$ and the
Radon--Nikod\'ym derivative is given by
\[
\frac{d \mathbb{P}_x^{\delta,c}|_{\mathscr{F}_t}}{d \mathbb
{P}_x^{\delta}|_{\mathscr{F}
_t}}=e^{-\alpha
t}\frac{\phi_{\alpha\uparrow}( X_t)}{\phi_{\alpha \uparrow}( x)}.
\]
From the form of the Radon--Nikod\'ym derivative, we see that the
finite-dimen\-sional distributions of the bridges of $\mathbb
{P}_x^{\delta,c}$
do not depend on $c$, they are just Bessel bridges. Therefore, we get
not only the existence of bridge laws but also their weak continuity
with respect to the parameters involved, because this is the case for $c=0$.

To end this subsection, let us mention a pathwise construction of
Bessel bridges from the trajectories of Bessel processes, contained in
Theorem 5.8 of \cite{pitmanYorBesselWithDrift}, page 324. It states
that the law of the bridge of a Bessel process (with or without drift)
from $x$ to $y$ of length $t$ can be obtained as:
\begin{longlist}
\item the law of $( uX_{1/u-1/t})_{u\in[0,t]}$ under
$\mathbb{P}_{y/t}^{\delta,\sqrt{2x}}$ or as
\item the law of
$( ( \frac{t-u}{t})X_{tu/( t-u)})_{u\in[0,t]}$ under
$\mathbb{P}_{x}^{\delta,\sqrt{y/t}}$.
\end{longlist}


\section[Construction of bridges]{Construction and weak continuity of Markovian bridges}\label{construction}
In this section, we will prove Theorem
\ref{bridgeByWeakCont} and Corollary
\ref{jointWeakContinuityBridgeLaws}. First, we discuss the heuristic of
our proof, which borrows heavily from the construction of Markovian
bridge laws of \cite{markovbridges}. Recall that we are working with a
Feller family $( \mathbb{P}_x)_{x\in S}$ on a LCCB space which
admits a
transition density $p_s( x,y)$ with respect to a $\sigma$-finite
measure $\mu$.

\subsection{Heuristics}
Let $0<s<t$ and note that for every $F\in b\mathscr{F}_s$ and every
$f\in
b\mathscr{B}_{S}$ the Markov property and the Tonelli--Fubini theorem
imply
\begin{eqnarray*}
\mathbb{E}_x\bigl( F\cdot f( X_t)\bigr)
&=&\mathbb{E}_x\bigl( F\cdot P_{t-s}f( X_s)\bigr)\\
&=&\int f( y)\mathbb{E}_x\bigl( F\cdot p _{t-s}( X_s,y)\bigr)\mu( dy).
\end{eqnarray*}
By restricting the last integral to
\[
\mathscr{P}_{t}=\{ y\in S\dvtx p_t( x,y)>0\},
\]
we obtain our base formula
\[
\mathbb{E}_x\bigl( F\cdot f( X_t)\bigr)
=\int_{\mathscr{P}_{t}} \mathbb{E}_x\biggl( F\cdot\frac{p _{t-s}(
X_s,y)}{p_t( x,y)}\biggr)f( y)p_t( x,y)\mu( dy).
\]
To construct a version of the conditional law of $( X_s)_{s\leq
t}$ given $X_t=y$ under $\mathbb{P}_x$, one could therefore seek to
build a
law $\mathbb{P}_{x,y}^t$ on the Skorohod space of c\`adl\`ag\
trajectories of $[0,t]$ into $S$, denoted $D_t$, such that for
every $s<t$, $\mathbb{P}_{x,y}^t$ is absolutely continuous with
respect to $\mathbb{P}_x$ with Radon--Nikod\'ym density $M_{x,y
}^s$ given by
%
\begin{equation}\label{bridgeLAC}
M_{x,y}^s=\frac{d\mathbb{P}_{x,y}^t|_{\mathscr{F}_s}}{d\mathbb{P}
_x|_{\mathscr{F}_s}}=\frac{p_{t-s}( X_s,y)}{p_t( x ,y)},
\end{equation}
because for such measures the equality
%
\begin{equation}\label{bridgeDesintegration}
\mathbb{E}_x\bigl( F\cdot f( X_t)\bigr)=\int_{\mathscr{P}_{t}} \mathbb{E}_{x,y}^t( F)f( y)p_{t}( x,y)\mu( dy)
\end{equation}
would follow for $s<t$. Equation (\ref{bridgeDesintegration}) contains
a disintegration of the law of $( X_r)_{r<s}$ with respect to
$X_t$ under $\mathbb{P}_x$. The laws $\mathbb{P}_{x,y}^t$ are usually called
bridges since under clearly stated hypotheses, the starting point condition
\[
\mathbb{P}_{x,y}^t( X_0=x)=1
\]
as well as the \textit{ending point condition}
%
\begin{equation}\label{endPoint}
\mathbb{P}_{x,y}^t( X_{t-}=X_t=y)=1
\end{equation}
are satisfied. This explains why, even if we succeed at constructing
such a law $\mathbb{P}_{x,y}^t$, the local absolute continuity
relationship (\ref{bridgeLAC}) would not hold for $s=t$, unless of
course the law of $X_t$ under $\mathbb{P}_x$ charges $y$ and for the
examples we have considered this is not the case. However, if we can
build the laws $\mathbb{P}_{x,y}^t$ satisfying the local absolute
continuity relationship (\ref{bridgeLAC}) and the ending point
condition (\ref{endPoint}) we can extend~(\ref{bridgeDesintegration})
to $s=t$ by the following argument. Let $\sigma_t\dvtx\bigcup _{s>t}D_s\to
D_t$ be defined by
\[
\sigma_t f( s)=
\cases{
f( s),&\quad if $s<t$,\cr
f( t-),&\quad if $s\geq t$.
}
\]
Then the ending point condition (\ref{endPoint}) implies that for
every $F\in b\mathscr{F}_t$,
\[
\mathbb{P}_{x,y}^t( F=F\circ\sigma_t)=1
\]
and, since Feller processes do not jump at fixed times,
\[
\mathbb{P}_x( F=F\circ\sigma_t)=1.
\]
The disintegration (\ref{bridgeDesintegration}) can be extended to
$\mathscr{F}
_{t-}=\sigma( X_s\dvtx s<t)$ by a monotone class argument and if $F\in
b\mathscr{F}_t$
then $F\circ\sigma_t\in b\mathscr{F}_{t-}$ so that
\begin{eqnarray*}
\mathbb{E}_x\bigl( Ff( X_t)\bigr)
&=&\mathbb{E}_x\bigl( F\circ\sigma_tf( X_t)\bigr)\\
&=&\int_{\mathscr{P}_{t}} \mathbb{E}_{x,y}^t( F\circ\sigma_t)f(y)p_{t}( x,y)\mu( dy)\\
&=&\int_{\mathscr{P}_{t}} \mathbb{E}_{x,y}^t( F)f( y)p_{t}( x,y)\mu(dy).
\end{eqnarray*}

To continue our discussion of bridges, recall that weak continuity of
the bridge laws is implied by tightness and weak continuity of
one-dimensional distributions. Weak continuity of one-dimensional
distributions is implied by continuity in variation, which is implied
by continuity of the densities by Scheffe's lemma. Hence, hypotheses
(H1) and (H2) are not far fetched. Together, they imply the
weak-continuity of finite dimensional distributions, at least for times
$s<t$, since the first implies the almost sure convergence
\[
M_{x,z}^s\to M_{x,y}^s
\]
as $z\to y$ under $\mathbb{P}_x$, and the second one implies the
applicability of Scheffe's lemma, since it implies that the integral of
$M_{x,y}^s$ with respect to $\mathbb{P}_x$ is equal to $1$.
Hypothesis (H3) does not have a simple explanation but its use is
very transparent in the proof of Theorem \ref{bridgeByWeakCont}.

\subsection{The proof}
Under the set of hypotheses (H1)--(H3) we will prove the next
theorem, which by the preceding discussion proves Theorem
\ref{bridgeByWeakCont}.

\begin{teorema}\label{bridgeExistence}
On $D_t$, the laws
$\mathbb{P}_x(\cdot\vert X_t\in
B_{\delta}( y))$ converge
weakly as
$\delta\to0$ to a law $\mathbb{P}_{x,y}^t$ which satisfies the
following three
conditions:
\begin{longlist}
\item[\textup{(1)}]\hypertarget{thmLAC} the local absolute
continuity relationship (\ref{bridgeLAC}),
\item[\textup{(2)}] the ending point condition (\ref{endPoint}) and
\item[\textup{(3)}] $y\mapsto\mathbb{P}_{x,y}^t$ is
weakly continuous.
\end{longlist}
\end{teorema}

\begin{pf}
The weak convergence statement will be proved in the usual manner, by
establishing tightness and the convergence of the finite-dimensional
distributions, although some technical preliminaries are needed.

Let us first see that the support of $\mu$ is $S$: let $y
\in
S$ and consider $\delta>0$. Then, there exists $t>0$ such
that
\[
\mathbb{P}_{y}\bigl( X_t\in B_{\delta}( y)\bigr)>0
\]
since $X_t$ converges in probability to $y$ as $t\to0$ under $\mathbb{P}
_y$, because of the Feller property. Since
\[
\mathbb{P}_{y}\bigl( X_t\in B_{\delta}( y)\bigr)=\int_{B_{\delta }( y)} p_t(
y,z)\mu( dz),
\]
it follows that $\mu( B_{\delta}( y))>0$.

Now we will obtain the approximation
%
\begin{equation}\label{densityApproximation}
\lim_{\delta\to
0,z\to y}\frac{\mathbb{P}_x( X_s\in B_{\delta }( z))}{\mu(
B_{\delta}( z))}
=p_s( x,y)
\end{equation}
of the transition density $p_s$. Since $p_s( x,\cdot )$ is continuous
at $y$, for every $\varepsilon>0$ there exists $\delta>0$
such that $|p_s( x,y)-p_s( x ,z)|<\varepsilon$ for all $z\in B_{\delta
}( y)$.
Therefore, for all $\delta'<\delta/2$ and all $z\in B_{\delta/2}( y)$:
\[
\biggl|p_s( x,y)-\frac{1}{\mu( B_{\delta'}( z))}\int_{B_{\delta'}( z)} p
_s( x,z')\mu( dz)\biggr|<\varepsilon,
\]
so that (\ref{densityApproximation}) holds.

The next step is to note that if $y\in\mathscr{P}_{t}$ then for all
$\delta>0$,
\[
\mathbb{P}_x\bigl( X_t\in B_{\delta}( y)\bigr)>0.
\]
This is because, by hypothesis (H1), there exists $\delta_0$ such
that $p_t( x,z)>0$ for all $z\in B_{\delta_0}( y)$. Therefore, for all
$\delta\leq\delta_0$,
\[
\mathbb{P}_x\bigl( X_t\in B_{\delta}( y)\bigr)=\int_{B_{\delta}( y)} p _t(x,z)\mu( dz)>0
\]
since otherwise, $\mu( B_{\delta}( y))=0$.

We will now take care of property \hyperlink{thmLAC}{(1)}. For any $F\in
b\mathscr{F}_s$
where $s<t$, the Markov property implies the equality
\[
\mathbb{E}_x\bigl(F\vert X_t\in B_{\delta
}( y)\bigr)
=\mathbb{E}_y\biggl( F\cdot\frac{\mathbb{P}_{X_s}( X_{t-s}\in B_{\delta
}( y))}{\mathbb{P}_x( X_t\in B_{\delta}( y))}\biggr),
\]
the right-hand side of which converges to
\[
\mathbb{E}_x\biggl( F\cdot\frac{p_{t-s}( X_s,y)}{p_t( x,y)}\biggr)
\]
because of (\ref{densityApproximation}) and Scheffe's lemma. The
latter is applicable because of the Chapman--Kolmogorov equations. From
this, we conclude something quite a bit stronger than the convergence
of finite-dimensional distributions: for any $s<t$, the law of $(
X_r)_{r\leq s}$ converges in variation (hence weakly) to a law $\mathbb{P}
_{x,y}^{t,s}$ on $D_s$ such that
\[
\mathbb{P}_{x,y}^{t,s}( A)=\mathbb{E}_{x}\biggl( \mathbf{1}_{A}\cdot
\frac{p_{t-s}( X_s,y)}{p_t( x,y)}\biggr).
\]
In particular, if $\tilde\omega( f,t,h)$ denotes the so-called
modified modulus
of continuity on $D_t$ given by
\[
\tilde\omega( f,t,h)=\inf_{\{ t_i\}}\max_i\max_{s,s'\in
[t_{i-1},t_i)}\rho( f( s),f( s')),
\]
where the infimum extends over all partitions
\[
0=t_0<t_1<\cdots<t_n=t
\]
such that $t_i-t_{i-1}>h$, then the above functional weak convergence
implies the following condition: for all $\varepsilon>0$ and $s<t$,
%
\begin{equation}\label{partialTightness}
\lim_{h\to0}\limsup_{\delta\to
0}\mathbb{P}_x\bigl(\tilde\omega(
X,s,h)>\varepsilon| X_t\in B_{\delta }( y)\bigr)=0.
\end{equation}
We will use (\ref{partialTightness}) to study the tightness of our
approximations
\[
\mathbb{P}_x\bigl(\cdot\vert X_t\in
B_{\delta}( y)\bigr)
\]
as $\delta\to0$. Let $Z_{h}=\sup_{s,s'\in[0,h]}\rho (
X_s,X_{s'})$. It suffices, in view of the convergence of
finite-dimensional distributions on $[0,s)$ and the fact that the law
of $X_t$ under the approximating law converges weakly to unit mass at
$y$ so that all finite-dimensional distributions converge, to
verify the following for all $\varepsilon>0$:
\[
\lim_{h\to0}\lim_{\delta\to0}\mathbb{P}_x\bigl(Z_{h}\circ \theta_{t-h}>\varepsilon\vert X_t\in
B_{\delta}( y)\bigr)=0.
\]
Not only will this prove weak convergence on $D_t$, but since $h$ is
included in the supremum defining $Z_h$, it will also prove the ending
point condition for the limit law.

To that end, we will now prove a technical result displayed in
(\ref{technicalConditionWeakContinuity}). By the Feller property, for
any compact set $K\subset S$, the laws $( \mathbb{P}_z )_{z\in
K}$ are weakly continuous on $D_h$ with respect to $z$. Since
for each individual law
\[
\lim_{h\to 0}\mathbb{P}_z( Z_h>\varepsilon)= 0
\]
and $z\mapsto\mathbb{P}_z( Z_h>\varepsilon)$ is continuous (because Feller
processes do not jump at fixed times and $Z_h$ seen as a functional on
$D_\infty$ is continuous at $f$ if $f$ is continuous at $h$) and
increasing in $h$, then
%
\begin{equation}\label{technicalConditionWeakContinuity}
\lim_{h\to0}\sup_{z\in K}\mathbb{P}_z( Z_h>\varepsilon)=0.
\end{equation}
Otherwise, there would be two sequences, $( z_n)$ in $K$
and $( h_n)$ decreasing to zero, such that
\[
\liminf_{n\to0}\mathbb{P}_{z_n}( Z_{h_n}>\varepsilon)>0.
\]
However, since $K$ is compact, there exists a subsequence $( z_{n_k})$
converging to $z\in K$ and because Feller processes do not
admit fixed-time discontinuities and have c\`adl\`ag paths
\begin{eqnarray*}
0&<&\liminf_{k\to\infty}\mathbb{P}_{z_{n_k}}( Z_{h_{n_k}}>\varepsilon)
\leq\liminf_{m\to\infty}\lim_{k\to\infty}\mathbb{P}_{z _{n_k}}(Z_{h_m}>\varepsilon)\\
&=& \lim_{m\to\infty}\mathbb{P}_{z}( Z_{h_m}>\varepsilon)=0,
\end{eqnarray*}
which is a contradiction.

To continue our main line of argument, note that by local compactness,
there exists a $\delta>0$ such that $B_{\delta}( y)$ has compact
closure. We will write
\begin{eqnarray*}
&&\mathbb{P}_x\bigl(Z_{h}\circ \theta_{t-h}>\varepsilon\vert X_t\in B_{\delta}( y)\bigr)\\
&&\qquad=\mathbb{P}_x\bigl(Z_{h}\circ\theta_{t-h}>\varepsilon,X_{t-h}\in B_{\delta}( y)\vert X_t\in B_{\delta}( y)\bigr)\\
&&\qquad\quad{}+\mathbb{P}_x\bigl(Z_{h}\circ\theta_{t-h}>\varepsilon,X_{t-h}\notin B_{\delta}( y)\vert X_t\in B_{\delta}( y)\bigr)
\end{eqnarray*}
and bound each one of the summands of the right-hand side. For the
first one, use Bayes rule
\begin{eqnarray*}
&&\mathbb{P}_x\bigl(Z_{h}\circ \theta
_{t-h}>\varepsilon,X_{t-h}\in B_{\delta}( y)\vert X_t\in B_{\delta}( y)\bigr)\\
&&\qquad=\mathbb{P}_x\bigl(Z_{h}\circ\theta_{t-h}>\varepsilon,X_{t}\in B_{\delta}( y)\vert X_{t-h}\in B_{\delta}( y)\bigr)
\frac{\mathbb{P}_x( X_{t-h}\in B_{\delta}( y))}{\mathbb{P} _x(X_t\in B_{\delta}( y))}.
\end{eqnarray*}
However, in view of the Markov property, the technical result of the
last paragraph, hypothesis (H3) and the transition density
approximation (\ref{densityApproximation}):
\begin{eqnarray*}
&&\mathbb{P}_x\bigl(Z_{h}\circ\theta_{t-h}>\varepsilon,X_{t}\in B_{\delta}( y)\vert X_{t-h}\in B_{\delta}( y)\bigr)\\
&&\qquad\leq\frac{\mathbb{P}_x( Z_{h}\circ\theta_{t-h}>\varepsilon,X_{t-h}\in B_{\delta}( y))}{\mathbb{P}_x( X_{t-h}\in B_{\delta}( y))}\\
&&\qquad\leq\sup_{z\in B_{\delta}( y)}\mathbb{P}_z ( Z_h>\varepsilon)
\end{eqnarray*}
and
\[
\lim_{h\to0}\lim_{\delta\to0}\frac{\mathbb{P}_x( X_{t-h}\in B_{\delta}( y))}{\mathbb{P} _x(X_t\in B_{\delta}( y))}=\lim_{h\to0}\frac{p_{t-h}( x,y)}{p_t( x,y)}=1,
\]
so that
\[
\lim_{h\to0}\limsup_{\delta\to0} \mathbb{P}_x\bigl(Z_{h}\circ \theta_{t-h}>\varepsilon,X_{t-h}\in
B_{\delta}( y)\vert X_t\in B_{\delta}( y)\bigr)=0.
\]
We will now obtain a second bound by means of
\begin{eqnarray*}
&&\mathbb{P}_x\bigl(Z_{h}\circ \theta_{t-h}>\varepsilon,X_{t-h}\notin B_{\delta}( y)\vert X_t\in B_{\delta}( y)\bigr)\\
&&\qquad\leq\mathbb{P}_x\bigl(X_{t-h}\notin B_{\delta}( y )\vert X_t\in B_{\delta}( y)\bigr)\\
&&\qquad=1-\frac{\mathbb{P}_x( X_{t-h}\in B_{\delta}( y),X_t\in B_{\delta}( y))}{\mathbb{P}_x( X_{t-h}\in B_{\delta }( y))}
\frac{\mathbb{P}_x( X_{t-h}\in B_{\delta}( y))}{\mathbb{P} _x(X_t\in B_{\delta}( y))};
\end{eqnarray*}
we have already seen that if $\delta\to0$ and we then let $h\to0$,
the second factor in the right-hand side of the last equality converges
to $1$. To study the first factor, write it as
\[
1-\frac{\mathbb{P}_x( X_{t-h}\in B_{\delta}( y),X_t\notin B_{\delta}( y))}
{\mathbb{P}_x( X_{t-h}\in B_{\delta}( y))}
\]
and use the Feller property in the following manner: for $\delta$
small enough [so that $B_{\delta}( y)$ has compact closure]
and $\delta'\in(0,\delta)$, let $\phi\dvtx S\to[0,1]$ be a
continuous function which is equal to $1$ on $B_{\delta'}( y)$
and vanishes outside $B_{\delta}( y)$, since $\phi$ is continuous
and vanishes at infinity, the Feller property implies that for all
$z\in B_{\delta'}( y)$
\[
\mathbb{P}_z\bigl( X_{h}\notin B_{\delta}( y)\bigr)\leq
\mathbb{E}_z\bigl( 1-\phi( X_h)\bigr)=|\mathbb{E}_z( \phi( X_h))-\phi(
z)|\leq
\Vert P_h-\operatorname{Id}\Vert.
\]
Since the previous estimation does not depend on $\delta'<\delta$,
our conclusion is that it holds for all $z\in B_{\delta }( y)$ and so,
by the Markov property,
\[
\frac{\mathbb{P}_x( X_{t-h}\in B_{\delta}( y),X_t\notin B_{\delta
}( y))}
{\mathbb{P}_x( X_{t-h}\in B_{\delta}( y))}\leq\|P_h-\operatorname
{Id}\|.
\]
We finally obtain
\[
\lim_{h\to0}\lim_{\delta\to0}\mathbb{P}_x\bigl(Z_{h}\circ \theta_{t-h}>\varepsilon,X_{t-h}\notin
B_{\delta}( y)\vert X_t\in B_{\delta}( y)\bigr)=0,
\]
which implies the existence of a law $\mathbb{P}_{x,y}^t$ on $D_t$ to
which $\mathbb{P}_x(\cdot\vert X_t\in
B_{\delta}( y))$
converges weakly as $\delta\to0$. As we have already remarked,
$\mathbb{P}
_{x,y}^t$ satisfies the local absolute continuity relationship (\ref{bridgeLAC}). It also satisfies the ending point condition since the law
of $X_t$ under $\mathbb{P}_x$ conditionally on $\{ X_t\in B_{\delta
}( y)\}$ is concentrated on $B_{\delta}( y)$.

To conclude the proof of the theorem, we must examine the weak
continuity of $\mathbb{P}_{x,y}^t$ as $y$ varies. To do it, we will
prove that if $K\subset S$ is compact in $\mathscr{P}_{t}$ then
\[
\bigl( \mathbb{P}_x\bigl(\cdot\vert X_t\in
B_{\delta}( z )\bigr)\bigr)_{z\in K,\delta>0}
\]
is tight in $D_t$. If this is true then $( \mathbb{P}_{x,z }^t)_{z\in
K}$ will be tight and because as $z\to y\in
\mathscr{P}_{t}$, $\mathbb{P}_{x,z}^t$ converges in variation to
$\mathbb{P}_{x
,y}^t$ on $D_s$ and the ending point condition is satisfied,
then the finite-dimensional distributions of $\mathbb{P}_{x,z}^t$
converge to those of $\mathbb{P}_{x,y}^t$ and therefore, there is
also weak convergence. To analyze the tightness of $( \mathbb{P} _x(\cdot\vert X_t\in B_{\delta}(
z)))_{z\in K,\delta>0}$, we note that tightness holds on $D_s$ for each
$s<t$, so that it suffices to prove, for all $\varepsilon>0$,
\[
\lim_{h\to0}\lim_{\delta\to0,z\to
y}\mathbb{P}_x\bigl(Z_{h}\circ \theta
_{t-h}>\varepsilon\vert X_t\in B_{\delta}( z)\bigr)=0.
\]
Our previous arguments can be extended to this case, since by the
density approximation (\ref{densityApproximation}),
\[
\lim_{h\to0}\lim_{\delta\to
0,z\to y}\frac{\mathbb{P}_x( X_{t-h}\in B_{\delta }( z))}{\mathbb
{P}_x( X_{t}\in B_{\delta}( z))}=1
\]
and for sufficiently small $\delta$ [so that $B_{2\delta}( y)$ has
compact closure] and $z\in B_{\delta}( y)$, we have that
\[
\lim_{h\to
0}\sup_{z'\in B_{\delta}( z)}\mathbb{P}_{z '}\bigl( X_h\notin
B_{\delta}( z)\bigr)\leq\lim_{h\to
0}\sup_{z'\in B_{2\delta}( y)}\mathbb{P}_{z '}\bigl( X_h\notin
B_{2\delta}( y)\bigr)=0
\]
by (\ref{technicalConditionWeakContinuity}) and
\[
\mathbb{P}_x\bigl(X_{t}\notin
B_{\delta}( z )\vert X_{t-h}\in B_{\delta}( z)\bigr)\leq
\Vert P_h-\operatorname{Id}\Vert.
\]
\upqed\end{pf}

We have used the Feller property in particular to ensure that bridge
laws reach their end-point continuously. This is important mainly
because of a defect of the Skorohod topology: if the limit of a
sequence jumps at the end of the interval, the sequence itself must
feature a jump at the same time. In general, when we condition a Markov
process to be around $y$ at time $t$, it will not jump exactly at time
$t$. If the bridge were discontinuous at the endpoint, there would then
be no possibility of tightness. The following example shows the
importance of the Feller property for tightness of conditional laws.

Let $v\dvtx[0,1)\to\mathbb{R}_+$ be integrable on closed intervals of
$[0,1)$ but not integrable on $[0,1)$. Consider the following Markov
process, which lives on the state space $\{ 1\}\times
[0,1)\cup\{ 2\}\times\mathbb{R}$ on which we place the metric $\rho
$ given
by
\[
\rho( ( i,x),( j,y))=|i-j|+|x-y|.
\]
When started at $( 1,x)$, the first coordinate stays at $1$ and
the second undergoes an uniform motion to the right jumping to $( 2,0)$
at rate $v$. On $\{ 2\}\times\mathbb{R}$, the first coordinate
stays at two while the second undergoes a Brownian motion. Call its
measure when it starts at zero $\mathbb{P}$. Because of the conditions imposed
on $v$, the uniform motion is constrained on $\{ 1\}\times[0,1)$ and
starting from $( 1,0)$ the jump time has density $f$ given by
\[
f( s)=v( s)e^{-\int_0^s v( r) \,dr}.
\]
For any $y\neq0$ and any $s<1$, we can build the law $\mathbb
{P}_{y,s}$ under
which the canonical process does uniform motion to the right on $\{ 1\}
\times[0,1)$ until time $s$ and then jumps to $( 2,0)$ where
the second coordinate does a Brownian bridge from $0$ to $y$ of length
$1-s$. Then the measures
\[
\mathbb{P}_{( 0,1),( 2,y)}^1=\int_0^1 \mathbb{P}_{y,s}f( s) \,ds
\]
disintegrate $\mathbb{P}|_{\mathscr{F}_1}$ with respect to $X_1$.
[Note that the law
of $X_1$ under $\mathbb{P}$ assigns zero mass to $( 2,0)$, so that to
disintegrate $\mathbb{P}$, we do not need to define $\mathbb
{P}_{0}$.] Under $\mathbb{P}
_{( 0,1),( 2,y)}^1$, the density of the jump time to
$( 2,0)$ is proportional to
\[
f( s)p_{1-s}( 0,y)=f( s)\frac{1}{\sqrt{2\pi
( 1-s)}}e^{-y^2/2(1-s)}.
\]
If we now chose
\[
v( x)=\frac{1}{2\sqrt{1-s}( 1-\sqrt{1-s})},
\]
then $f( s)=1/2\sqrt{1-s}$, and we see that the law of the
unique jump time under $\mathbb{P}_{( 0,1),( 2,y)}^1$ converges to
$1$ as $y\to0$ since $1/(1-s)$ is not integrable on $[0,1)$. We
conclude that if $y_n\to0$, $( \mathbb{P}_{( 0,1),( 2,y_n)}^1)_n$ is
not tight. Indeed,
\begin{eqnarray*}
&&\lim_{h\to0}\limsup_{n\to\infty}\mathbb{P}_{( 0,1),( 2,y_n)}^1\Bigl( \sup_{u,v\in [t-h,t)}\rho( X_u,X_v)>\varepsilon\Bigr)
\\
&&\qquad\geq\lim_{h\to0}\limsup_{n\to\infty}\mathbb{P}_{y_n}\bigl( X\mbox{ jumps in }[t-h,t)\bigr)\\
&&\qquad =1
\end{eqnarray*}
for any $\varepsilon<1$. By the same reason, the conditional law of
$X$ on
$[0,1]$ given $X_1\in B_{\delta}( ( 2,0))$ is not tight as
$\delta\to0$.

We now turn to weak continuity of bridge laws with respect to the
length and the starting point.

\begin{pf*}{Proof of Corollary \ref{jointWeakContinuityBridgeLaws}}
Let us prove that as $t'\to t$ and $z\to y$ (in $\mathscr{P}_{t}$),
$\mathbb{P}_{x,z}^{t'}$ converges in law to $\mathbb{P}_{x,y}^t$.
As in
the proof of Theorem \ref{bridgeExistence}, under (H1$'$) we have
convergence in variation of $\mathbb{P}_{x,z}^{t'}|_{\mathscr{F}_s}$ to
$\mathbb{P}_{x,y}^{t}|_{\mathscr{F}_s}$ if $s<t$ and, because of the
ending point
condition, this implies not only the convergence of the
finite-dimensional distributions but also a tightness criterion on the
compact intervals of $[0,\infty)\setminus\{ t\}$. Hence, we must only
prove the following for all $\varepsilon>0$:
\[
\lim_{h\to0}\lim_{\delta\to0,z\to y,t'\to t}\mathbb{P}_x\bigl(Z_{h}\circ\theta_{t'-h}>\varepsilon\vert X_{t'}\in B_{\delta}( z)\bigr)=0.
\]

Again, we can use the same arguments as in the proof of Theorem
\ref{bridgeExistence} since under (H1$'$),
(\ref{densityApproximation}) can be generalized to
\[
\lim_{\delta\to0,z\to y,s'\to s}\frac{\mathbb{P}_x( X_{s'}\in B_{\delta}( z))}{\mu( B_{\delta}( z))}=p_s( x,y).
\]
The other bounds needed did not depend on the length parameter $t'$.

We can extend the preceding reasoning by imposing the joint continuity
of the density in all variables to obtain the joint weak continuity of
bridge laws $\mathbb{P}_{x,y}^t$ in all variables.
\end{pf*}


\section{The backward strong Markov property}\label{backward}
In this section, we will prove Theorem \ref{backwardSMP}. We begin
with a basic summary of the properties of conditional independence
which we will need. We use the notation $\mathscr{G}_1\perp_{\mathscr
{H}}\mathscr{G}_2$
to mean that $\mathscr{G}_1$ and $\mathscr{G}_2$ are conditionally
independent given a
$\sigma$-field $\mathscr{H}$. 

\begin{pro}
The $\sigma$-fields $\mathscr{G}_1$ and $\mathscr{G}_2$ are
conditionally independent
given $\mathscr{H}$ if
and only if for all $G\in b \mathscr{G}_1$:
\[
\mathbb{E}( G | \mathscr{G}_2,\mathscr{H} )=\mathbb{E}( G |\mathscr{H} ).
\]
Furthermore, for any $\sigma$-fields $\mathscr{H},\mathscr
{G},\mathscr{G}_1,\mathscr{G}_2,\ldots\,$,
the following
conditions are equivalent:
\begin{longlist}[\textup{(ii)}]
\item[\textup{(i)}] ${\displaystyle\mathscr{G}\perp_{\mathscr{H}}\mathscr
{G}_1,\mathscr{G}_2,\ldots\,}$.
\item[\textup{(ii)}] For
any $n\geq1$,
\[
\mathscr{G}\perp_{\mathscr{H},\mathscr{G}_1,\ldots,\mathscr
{G}_n}\mathscr{G}_{n+1}.
\]
\end{longlist}
Finally, if $\mathscr{G}_1\perp_{\mathscr{H}}\mathscr{G}_2$ and
$\mathscr{G}'_1\subset\mathscr{G}_1$
then
\[
\mathscr{G}'_1\perp_{\mathscr{H}}\mathscr{G}_2\quad\mbox{and}\quad\mathscr{G}_1\perp_{\mathscr{H},\mathscr{G}'_1}\mathscr{G}_2
\]
\end{pro}

The first property is the \textit{asymmetric expression of conditional
independence} and is the link between conditional independence and the
Markov property, as has been expanded upon. The second of the above
properties will be referred to as the \textit{chain rule for
conditional independence}. Proofs of them are found in \cite{kallenberg}. The third
property consists of the \textit{downward monotone character of
conditional independence} in the nonconditioning $\sigma$-fields and a
\textit{partial upward monotone character in the conditioning
$\sigma $-field}. It
is a trivial application of the chain rule since under the conditions
stated, $\sigma( \mathscr{G}_1,\mathscr{G}'_1)=\sigma( \mathscr
{G}_1)$. We cannot expect a general
upward monotone character to hold: for example, if $X$ and $Y$ are two
independent random variables on $\{ -1,1\}$ which take the two values
with equal probability, and $Z=XY$, then $X$ and $Y$ are independent
but they are not conditionally independent given $Z$, since the
conditional law of $Y$ given $X,Z$ is concentrated at $XZ$ and the
conditional law of $Y$ given $Z$ is the same as that of $Y$ since $Y$
and $Z$ are independent. The following formulation of the preceding
example might be more impressive. Let $\mathscr{H}_1\subset\mathscr
{H}_2\subset\mathscr{H}_3$
then
\[
\mathscr{G}_1\perp_{\mathscr{H}_1}\mathscr
{G}_2\quad\mbox{and}\quad\mathscr{G}_1\perp_{\mathscr{H}_3}\mathscr{G}_2\quad\mbox{do not imply}\quad\mathscr{G}_1\perp_{\mathscr{H}_2}\mathscr{G}_2;
\]
just take $\mathscr{H}_1=\{ \Omega,\varnothing\}$, $\mathscr
{H}_2=\sigma( Z)$, $\mathscr{H}_3=\sigma( X,Y)$, $\mathscr
{G}_1=\sigma( X)$ and $\mathscr{G}_2=\sigma( Y)$.
During the course of the proof of the backward strong Markov property,
we will use the following.

\begin{definic*} Given \textit{two $\sigma$-fields} $\mathscr{G}$ and
$\mathscr{G}'$, we say that they \textit{agree on a set $A$}, written
$\mathscr{G}=\mathscr{G}
'$ on
$A$, if $A\in\mathscr{G}\cap\mathscr{G}'$ and $A\cap\mathscr
{G}=A\cap\mathscr{G}'$.
\end{definic*}

\begin{pro}[(Local property of conditional expectation)]
On a probability space $( \Omega,\mathscr{F},\mathbb{P})$, let
$\mathscr{G}$ and $\mathscr{G}'$ be sub-$\sigma
$-fields of
$\mathscr{F}$
and consider two integrable random variables $\xi,\xi'$. Suppose that
$\mathscr{G}=\mathscr{G}'$ on $A$ and that $\xi=\xi'$ almost surely
on $A$.
Then
\[
\mathbb{E}( \xi | \mathscr{G} )=\mathbb{E}( \xi'
| \mathscr{G}' )\qquad\mbox{almost surely on }A.
\]
\end{pro}

The preceding proposition is proved in \cite{kallenberg}.

\begin{pf*}{Proof of Theorem \ref{backwardSMP}}
We begin by discussing the measurability of the mapping
\[
( t,x,y)\mapsto\mathbb{P}_{x,y}^t( F)
\]
for any measurable $F\dvtx D_\infty\to\mathbb{R}$. First, let us note that
the set
\[
\{ ( t,x,y)\dvtx p_t( x,y)>0\}
\]
is measurable because $( t,x,y)\mapsto p_t( x ,y)$ is measurable since
it is jointly continuous in $( t,y)$ for fixed $x$ and measurable in
$x$ for fixed
$( t,y)$. The latter is true since for all $\delta>0$,
\[
x\mapsto\frac{\mathbb{P}_x( X_t\in B_{\delta}( y))}{\mu ( B_{\delta}( y))}
\]
is measurable by the measurability property of Markovian families and
its limit as $\delta\to0$ is $p_t( x,y)$ by the
density approximation (\ref{densityApproximation}) implied by
hypothesis (H1$'$).

For the rest of the argument, we will work on the set
$\{ ( t,x,y)\dvtx p_t( x,y)>0\}$. Let us
note that
if $F\in b\mathscr{F}_s$ and $s<t$, then the local absolute continuity
relationship (\ref{bridgeLAC}) implies that $x\mapsto
\mathbb{P}_{x,y}^t( F)$ is measurable and by the monotone class
theorem, we see that the measurability extends first to any $F\in
b\mathscr{F}_t$ and then to any measurable $F$. Since by Corollary
\ref{jointWeakContinuityBridgeLaws}, $( t,y)\mapsto
\mathbb{P}_{x,y}^t( F)$ is continuous if $F$ is, we see that
$( t,x,y)\mapsto\mathbb{P}_{x,y}^t( F)$ is measurable
whenever $F$ is continuous. By a monotone class argument, the preceding
measurability extends to measurable $F$.

We now turn to the computation of the conditional expectation of
Theorem \ref{backwardSMP}. Because of the strong Markov property, it
suffices to prove the theorem when $S=0$; we will simplify the notation
for $\sigma^0_L$ to $\sigma_L$.

Let
\[
L^n=\sum_{k=0}^{\infty}\frac{k}{2^n}\mathbf{1}_{(k/2^n,(k+1)/2^n]}( L).
\]
Then $L^n$ is a random time strictly smaller than $L$ which increases
with $n$ toward~$L$. Since $L$ is a backward optional time,
\[
\biggl\{ L^n=\frac{k}{2^n}\biggr\}=\biggl\{ \frac{k}{2^n}<L\leq \frac{k+1}{2^n}\biggr\}\in
\mathscr{F}^{k/2^n}.
\]
Furthermore, the $\sigma$-fields $\mathscr{F}^{k/2^n}$ and $\mathscr
{F}^{L^n}$ agree
on the set
$\{ L^n=k/2^n\}$ since $\theta_{L^n}$ coincides with $\theta
_{k/2^n}$ on that set. For every bounded and measurable $H\dvtx D _\infty
\to\mathbb{R}$
\[
\mathbb{E}_\nu(H\circ\sigma
_{k/2^n}\mathbf{1}_{L^n=k/2^n}\vert \mathscr{F} ^{k/2^n})=\mathbb
{P}_{X_0,X_{k/2^n}}^{k/2^n}( H)\mathbf{1}_{L^n=k/2^n},
\]
so that by the local property of conditional expectation:
%
\begin{equation}\label{approximationBSMP}
\mathbb{E}_\nu(H\circ\sigma_{L^n}\vert\mathscr{F}^{L^n})=\mathbb{P} _{X_0,X_{L^n}}^{L^n}( H)\qquad\mbox{a.s. on }\{ L^n>0\}.
\end{equation}
If $H$ is actually continuous and bounded, then
\[
H\circ\sigma_{L^n}\to H\circ\sigma_L.
\]
If $A\in\mathscr{F}^L$ and $B\in\mathscr{B}_{S}$ then $A\cap\{
X_{L-}\in B\}\cap
\{ L_n>0\}\in\mathscr{F}^{L^n}$, and so (\ref{approximationBSMP}) implies
\begin{eqnarray*}
&&\mathbb{E}_\nu( H\circ\sigma^{L^n}\mathbf{1}_{A}\mathbf{1}_{X_{L-}\in B}\mathbf{1}_{L^n>0})\\
&&\qquad=\mathbb{E}_\nu( \mathbf{1}_{A}\mathbf{1}_{X_{L-}\in B}\mathbb{P} _{X_0,X_{L^n-}}^{L^n}( H)\mathbf{1}_{L^n>0}).
\end{eqnarray*}
The left-hand side of the preceding expression converges to
\[
\mathbb{E}_\nu( H\circ\sigma^{L}\mathbf{1}_{A}\mathbf
{1}_{X_{L-}\in B}\mathbf{1}_{L>0})
\]
as $n\to\infty$, while the right-hand side converges to
\[
\mathbb{E}_\nu( \mathbf{1}_{A}\mathbf{1}_{X_{L-}\in B}\mathbb
{P}_{X_0,X_{L-}}^{L}( H)\mathbf{1}_{L>0})
\]
by Corollary \ref{jointWeakContinuityBridgeLaws}, so that
\[
\mathbb{E}_\nu(H\circ\sigma_{L}\vert \mathscr{F}^{L},X_{L-})=\mathbb{P} _{X_0,X_{L-}}^{L}( H)\qquad\mbox{a.s. on }\{ L>0\}.
\]
\upqed
\end{pf*}

\section[Pathwise constructions]{Self-similarity and pathwise construction of Markovian bridge laws}\label{pathwise}
In this section, we will discuss examples for which
the pathwise construction of bridges of self-similar Feller processes
of Theorem \ref{pathwiseTheorem} works and we will verify the pathwise
construction of the stable subordinator conditioned to die at a given
level of Theorem \ref{conditionedSubordinatorTheorem}. The latter is
found in Section \ref{selfSimBridges} while the former is included
in Section \ref{conditionedSubordinator}.

\subsection{Pathwise construction of bridges of self-similar Markov processes}\label{selfSimBridges}
Note that Theorem \ref{pathwiseTheorem} is
trivial from Theorem \ref{backwardSMP}; however, the real problem lies
in identifying processes for which the hypothesis holds. In this
section, we give several (general) examples and a word of caution
against the impression that the hypothesis should hold trivially
because of self-similarity.

\begin{example}\label{spectrallyAsymmetric}
Consider first a self-similar Feller
family $( \mathbb{P}_{x})_{x\in S}$ of index $\gamma$, and suppose
that under each $\mathbb{P}_x$, the jumps of $X$ have the same sign.
As we now
see, the set
\[
\mathscr{Z}_c=\{ t\in(0,1]\dvtx X_{t-}=ct^{1/\gamma}\}
\]
is almost surely not empty under $\mathbb{P}_0$ if $\mathbb{P}_0(
X_1>c)$ and
$\mathbb{P}_0( X_1<c)$ are both positive.

To see this, consider a sequence $( t_n)$ decreasing to zero and
define the set
\[
A=\limsup_{n\to\infty}\{ X_{t_n}\mbox{ or }X_{t_n-}>ct_n^{1/\gamma}\}.
\]
By Blumenthal's 0--1 law, $A$ is $\mathbb{P}_x$ trivial for every $x$. However,
under $\mathbb{P}_0$ we can apply scaling to give
\[
\mathbb{P}_0( A)\geq\limsup_{n\to\infty}\mathbb{P}_0(X_{t_n}\mbox{ or }X_{t_n-}>ct_n^{1/\gamma})=\mathbb{P}_0( X_1\mbox{ or }X_{1-}>c).
\]
We see that $\mathbb{P}_0( A)=1$ when $\mathbb{P}_0( X_1>c)$ is positive.
By the same argument, if
\[
B=\limsup_{n\to\infty}\{ X_{t_n}\mbox{ or }X_{t_n-}<ct_n^{1/\gamma}\}
\]
then $\mathbb{P}_0( B)=1$ when $\mathbb{P}_0( X_1<c)>0$. If $\mathbb
{P} _0( X_1>c)$ and $\mathbb{P}_0( X_1<c)$ are both positive then $X$ will
cross the curve $t\mapsto ct^{1/\gamma}$ an infinite number of times
near zero, in the sense that for every $t\in(0,1)$ there will exist
$s\in(0,t)$ such that $\operatorname{sgn}( X_s-ct^{1/\gamma})\neq
\operatorname{sgn} ( X_s-c s^{1/\gamma})$. However, either the
downcrossings or the
upcrossings will touch the curve, since $X$ either decreases or
increases continuously, which implies the existence of $t\in(0,1)$
such that $X_t=ct^{1/\gamma}=X_{t-}$.

This reasoning implies that Thereom \ref{pathwiseTheorem} holds for
Brownian motion and Bessel processes. It also holds for spectrally
asymmetric stable L\'evy processes: these are stable L\'evy processes
whose jumps have almost surely the same sign.
\end{example}

\begin{example}\label{cautionPathwiseTheorem}
We continue with the special case of
stable L\'evy processes which are not spectrally asymmetric with
following result.
\end{example}

\begin{teorema}\label{cautionResult}
Let $\mathbb{P}$ be the law of a stable L\'evy
process of
index $\alpha\in(0,2]$ started at $0$; then $\mathbb{P}( g_c>0)=1$ if and
only if either $\alpha>1$ or $\alpha<1$ and $c\neq0$.
\end{teorema}

When the process is spectrally asymmetric we use Example
\ref{spectrallyAsymmetric}. When the process has jumps of both signs,
our proof of Theorem \ref{cautionResult} passes through associated
Ornstein--Uhlenbeck process, this is the process $Y$ defined by
$Y_t=e^{-t/\alpha}X_{e^t}$ for $t\in\mathbb{R}$ under $\mathbb
{P}_0$. Then $Y$ is a
stationary (time-homogeneous) Markov process whose semigroup is
described as follows (cf. \cite{breiman} or \cite{stableWindings}): let
$f_t$ be the density of $X_t$ under $\mathbb{P}_0$ with respect to Lebesgue
measure (as in Section \ref{levyProcExample}) and set
\[
p_t( x,y)=f_t( y-x).
\]
Then the semigroup of $Y$ admits transition densities $q_t,t\geq0$,
given by
%
\begin{equation}\label{transitionDensityOU}
q_t( x,y)=p_{e^t-1}( x,e^{t/\alpha}y)e^{t/\alpha}.
\end{equation}
Note the equality
\[
\{ t>0\dvtx X_{t-}=ct^{1/\alpha}\}=\exp( \{ t\in\mathbb{R}\dvtx Y_{t-}=c\}).
\]
The positivity of $g_c$ under $\mathbb{P}_0$ would follow if the set
$\{ t\in\mathbb{R}\dvtx Y_{t-}=c\}$ had no lower bound almost surely;
since $Y$ is
Feller, it is sufficient to prove this for $A=\{ t\in\mathbb{R}\dvtx
Y_{t-}\mbox{ or }Y_t=c\}$, by using their quasi-continuity as in \cite{bertoinLevyProcesses}, Corollary 8, page 22. This would in turn be obtained
if the set $A\cap(0,\infty)$ were nonempty with positive probability,
since the stationary character of $Y$ under $\mathbb{P}_0$ implies that
\begin{eqnarray*}
\mathbb{P}( A\mbox{ has no lower bound} )&=&\lim_{M\to-\infty}\mathbb{P}\bigl( A\cap(-\infty,M)\neq\varnothing
\bigr)\\
&=&\mathbb{P}\bigl( A\cap(-\infty,M)\neq\varnothing \bigr),
\end{eqnarray*}
where the last equality is a consequence of the translation invariance
of $A$ under $\mathbb{P}_0$. This same translation invariance also
shows us that
\[
\mathbb{P}\bigl( A\cap(-\infty,M)\neq\varnothing \bigr)=\lim_{M\to\infty
}\mathbb{P}\bigl( A\cap (-\infty,M)\neq\varnothing \bigr)=\mathbb{P}\bigl( A\cap
(-\infty,\infty)\neq \varnothing \bigr)
\]
and the same reasoning gives us
\[
\mathbb{P}\bigl( A\cap(-\infty,\infty)\neq\varnothing \bigr)=\mathbb{P}\bigl(
A\cap (0,\infty)\neq\varnothing \bigr).
\]
In conclusion, $\{ A\mbox{ has no lower bound}\}$ and $\{ A\cap
(0,\infty)\neq\varnothing\}$ have the same probability.
Note also that
\[
\{ A\cap(0,\infty)\neq \varnothing\}=\{ \exists t>0, Y_{t-}\mbox{ or }Y_t=c\}
\]
and that if we let $T_c=\inf\{ t\geq0\dvtx Y_{t-}\mbox{ or }Y_t=c\}$,
the question of knowing whether $g_c$ is positive or not has been
recast as a question concerning the finitude of the random time $T_c$
for the associated Ornstein--Uhlenbeck process $Y$. The latter problem
can be solved explicitly by use of polarity criteria using the
resolvent density $v_\lambda$ of $Y$ given by
\[
v_\lambda( x,y)=\int_0^\infty e^{-\lambda t}q_t( x,y) \,dt
\]
for $\lambda>0$. Note that $v_\lambda$ can take the value $\infty$;
since $q_t$ is continuous, $v_\lambda$ is lower semicontinuous and so
it is continuous at $( x,y)$ (as a function with values on
$[0,\infty]$) if $v_\lambda( x,y)=\infty$. We will see the
following proposition.

\begin{pro}\label{finitudeOfResolventDensity}
When the stable process has jumps of
both signs, $v_\lambda$ is bicontinuous and
\[
v_\lambda( x,y)<\infty\quad\Leftrightarrow\quad
\cases{
\alpha\in(1,2),\cr
\alpha=1\mbox{ and }x\neq y\mbox{ or}\cr
\alpha\in(0,1)\mbox{ and $x$ or $y$ are not zero.}
}
\]
\end{pro}

Theorem \ref{cautionResult} follows from Proposition
\ref{finitudeOfResolventDensity} by the following well-known method:
let $( \mathbb{Q}_x)_{x\in\mathbb{R}}$ be the Markovian family
associated to the
semigroup of $Y$, introduce the resolvent operator defined
by
\[
V_\lambda( x,A)=\int_0^\infty e^{-\lambda t}\mathbb{Q}_x( X_t\in A)\, dt=\int_{A}v_\lambda( x,z)\, dz,
\]
as well as the stopping time
\[
H_\varepsilon=\inf\{ t\geq0\dvtx X_t\in B_{\varepsilon}( c)\},
\]
so that by the strong Markov property,
\[
V_\lambda( x,B_{\varepsilon}( c))=\mathbb{E}_{\mathbb{Q}_x}(e^{-\lambda H_\varepsilon}V_\lambda( X_{H_\varepsilon},B_{\varepsilon}( c))).
\]
Note that as $\varepsilon\to0$, $H_\varepsilon$ converges to $T_c$.

If $x\neq c$, then in any case $v_\lambda( x,c)<\infty$
and
\[
v_{\lambda}( x,c)=\lim_{\varepsilon\to0}\frac{1}{2\varepsilon}V_\lambda( x,B_{\varepsilon}( c)).
\]
If $v_\lambda( c,c)<\infty$, then
\[
v_\lambda( c,c)=\lim_{\varepsilon\to0}\frac{1}{2\varepsilon}V_\lambda( X_{H_\varepsilon},B_{\varepsilon}( c))
\]
and the bounded convergence theorem tells us that
\[
v_\lambda( x,c)=\mathbb{E}_{\mathbb{Q}_x}( e^{-\lambda T_c})v_\lambda( c,c),
\]
so that
\[
\mathbb{E}_{\mathbb{Q}_x}( e^{-\lambda T_c})>0
\]
implying the almost sure finitude of $T_c$ under $\mathbb{Q}_x$ for
any $x\neq
c$ so that
\[
\mathbb{P}( \exists t>0\dvtx Y_{t-}\mbox{ or }Y_t=c )=\int\mathbb{Q}_x( T_c<\infty)f_1( x)\, dx=1.
\]

If, on the other hand, $v_\lambda( c,c)=\infty$, then Fatou's
lemma tells us that
\[
\mathbb{E}_{\mathbb{Q}_x}( \infty\cdot e^{-\lambda T_c} )=\mathbb{E}_{\mathbb{Q}_x}\biggl( \liminf_{\varepsilon
\to0}e^{-\lambda H_\varepsilon}\frac{1}{2\varepsilon}V_\lambda(
X_{H_\varepsilon},B_{\varepsilon }( c))\biggr) \leq
v_\lambda( x,c)<\infty,
\]
so that $\mathbb{Q}_x( T_c=\infty)=1$ for all $x\neq c$ and so
\[
\mathbb{P}( \exists t> 0\dvtx Y_{t-}\mbox{ or }Y_t=c )=\int\mathbb{Q}_x( T_c<\infty)f_1( x)\, dx=0.
\]

Even though the proof of Proposition \ref{finitudeOfResolventDensity}
requires only elementary analysis, it is long and technical; it is
therefore presented in the \hyperref[app]{Appendix}.

\begin{note*}
It is also a consequence of Proposition
\ref{finitudeOfResolventDensity} that for the Ornstein--Uhlenbeck
process driven by an $\alpha$-stable L\'evy process which has jumps of
both signs, $x$ is polar if and only if $\alpha=1$ or $\alpha\in(0,1)$
and $x=0$. Also, the resolvent density $v_\lambda$ has been explicitly
computed by Patie in \cite{PatieTwoSided} in the spectrally asymmetric
case of index $\alpha\in(1,2)$. It is expressed in terms of Novikov's
generalization of Hermite's function introduced in
\cite{novikovAlphaHermite}.
\end{note*}

Suppose now that $\mathbb{P}$ is the law of Brownian motion. We now
study the
law of the random variable $g_c$ by proving Proposition
\ref{HermiteProposition}.
\begin{pf*}{Proof of Proposition \ref{HermiteProposition}}
Again, we make use of the stationary Ornstein--Uhlenbeck process
$Y_t=e^{-t/2}X_{e^t}$. Then
\[
g_c=e^{-\inf\{ r\geq0:Y_{-r}=c\}}.
\]
By time inversion, we see that $( Y_{-t})_{t\in\mathbb{R}}$ has the
same law as $( Y_{t})_{t\in\mathbb{R}}$, so that
\[
g_c\mbox{ has the same law as } \log( \inf\{ t\geq0\dvtx Y_t=c\}).
\]

Let $\mathbb{Q}_x,x\in\mathbb{R}$, stand for the Feller family of
$Y$. The (extended)
generator $A$ of $Y$ is given by
\[
Af( x)=\frac{d}{dt}\biggr|_{t=0}\mathbb{E}_{\mathbb{Q}_x}( f( X_t))=\frac{f''( x)}{2}-\frac{1}{2}xf'( x)
\]
if $f\dvtx \mathbb{R}\to\mathbb{R}$ has two bounded continuous
derivatives. The
Laplace transform of the hitting times $T_c$ of $c$ under $\mathbb
{Q}_x$ can
be expressed in terms of monotone eigenfunctions of the preceding
generator, which are in turn, expressible in terms of parabolic
cylinder functions; integrating with respect to the law of $Y_0$ will
then give us an expression of the Mellin transform of $g_c$. We now
provide a streamlined exposition of this suited to our needs.

Consider the nonnegative function on $\mathbb{R}$ given by
%
\begin{equation}\label{HeDefinition}
H_q( x)=\int_0^\infty e^{-x z-z^2/2}z^{q-1};
\end{equation}
integrating by parts in the preceding expression leads to
\[
H_{q+2}( x)=qH_q( x)-xH_{q+1}( x)
\]
while differentiating under the integral in (\ref{HeDefinition}) gives
\[
H_q'( x)=-H_{q+1}( x),
\]
so that $H_q$ is decreasing and
\[
H_{2q}''( x)-xH_{2q}'( x)=2q H_{2q}( x).
\]
The same equation is satisfied by $x\mapsto H_{2q}( -x)$. It\^o's
formula then tells us that the processes $e^{-q t}H_{2q}( X_t)$
and $e^{-q t}H_{2q}( -X_t)$ (for $t\geq0$) are local martingales
under $\mathbb{P}_x$ for any $x$, the first one of which is bounded up
to time
$T_c$ if $x\geq c$ while the second one is bounded up to $T_c$ if
$x\leq c$. By optional stopping, we see that
\[
\mathbb{E}_{\mathbb{Q}_x}( e^{-q T_c})=\frac{H_{2q}( x\operatorname{sgn}( x-c))}{H_{2q}( c\operatorname{sgn}( x-c))}.
\]
Hence,
%
\begin{equation}\label{MellinTransformGc}
\mathbb{E}( g_c^q )=\int_{-\infty}^{\infty}\frac{e^{-x^2/2}}{\sqrt{2\pi}}\frac{H_{2q}( x\operatorname{sgn}( x-c))}{H_{2q}( c\operatorname{sgn} ( x-c))}\,dx.
\end{equation}
To evaluate the last integral, use (\ref{HeDefinition}) to obtain
\[
\int_c^{\infty}\frac{e^{-x^2/2}}{\sqrt{2\pi}}\frac{H_{2q}(x)}{H_{2q}( c)}\,dx=\frac{e^{-c^2/2}}{\sqrt{2\pi}2q}\frac{H_{2q+1}( c)}{H_{2q}( c)}
\]
so that
%
\begin{equation}\label{firstEquationForMoments}
\mathbb{E}( g_c^q )=\frac{e^{-c^2/2}}{\sqrt{2\pi}2q}\biggl( \frac{H_{2q+1}( c)}{H_{2q}( c)}+\frac{H_{2q+1}( -c)}{H_{2q}( -c)}\biggr).
\end{equation}
To obtain the required result, recall that the Wronskian $W$ of the two
solutions $c\mapsto H_q( c)$ and $c\mapsto H_q( -c)$ of the
differential equation $f''( c)-cf'( c)-qf( c)=0$ is
given by Abel's identity
\[
W( c)=W( 0)e^{-c^2/2}.
\]
However, since $H_q'( c)=-H_{q+1}( c)$, $W$ can also be
expressed as
\[
W( c)=H_q( c)H_{q+1}( -c)+H_{q}( -c)H_{q+1}( c).
\]
Substituting in (\ref{firstEquationForMoments}), we get the following
expression expression for the $q$th moment of $g_c$:
\[
\frac{1}{2q\sqrt{2\pi}}W_{2q}( 0)\frac{1}{H_{2q}( c)H_{2q}( -c)}.
\]
Since
\[
H_q( 0)=2^{q/2-1}\Gamma\biggl( \frac{q}{2}\biggr),
\]
the quantity $W( 0)$ is explicitly evaluated as follows:
\[
W( 0)=2H_{2q}( 0)H_{2q+1}( 0)=2^{2q-1/2}\Gamma \biggl( \frac{q}{2}\biggr)\Gamma\biggl(\frac{q}{2}+\frac{1}{2}\biggr)=\sqrt{2\pi}\Gamma( 2q),
\]
where we have used the duplication formula for the $\Gamma$ function
in the last equality. This gives
\[
\mathbb{E}( g_c^q )=\frac{\Gamma( 2q)}{2qH_{2q}( c)H_{2q}( -c)}.
\]
For the asymptotic behavior of the moments of $g_c$, it suffices to
apply Laplace's method (as found in \cite{deBruijn} or \cite{olver})
to obtain, as $q\to\infty$,
\[
H_q( x)\sim\sqrt{\pi} q^{q/2}e^{-x\sqrt{q}-q/2}.
\]
\upqed\end{pf*}


\subsection{Pathwise construction of the stable subordinator conditioned to die at a given level}\label{conditionedSubordinator}
The aim of this subsection is to prove
Theorem \ref{conditionedSubordinatorTheorem}.

First of all, note that the effect of the scaling operator $S^\alpha_v$
to the stable subordinator started at zero and conditioned to die at
$b>0$ gives the stable subordinator started at zero and conditioned to
die at $v^{1/\alpha}b$; this is proved by the scaling properties of
$\mathbb{P}_0^\alpha$ and its relationship to $\mathbb{P}_0^{h_\alpha}$.

We need four additional elements to verify the desired pathwise
construction:
\begin{longlist}
\item[1.] The law of $\zeta$ under
$\mathbb{P}_a^{h_\alpha}$. This is obtained from the fact that
\begin{eqnarray*}
\mathbb{P}_a^{h_\alpha}( \zeta>t)&=&\frac{1}{h_\alpha( a)}\mathbb{E}_a^{\alpha}( h_\alpha( X_t))\\
&=&\frac{1}{h_\alpha( a)}\int f^\alpha_t( x)h_\alpha( x) \,dx\\
&=&\frac{1}{h_\alpha( a)}\int_t^\infty f^\alpha _s( b-a)\, ds
\end{eqnarray*}
where the last equality follows by the definition of $h_\alpha$ and
the Chapman--Kolmogorov equations. The density of $\zeta$ under
$\mathbb{P}
_a^{h_\alpha}$ is then equal to
\[
t\mapsto\frac{f_t( b-a)}{h_\alpha( a)}.
\]
\item[2.] The computation of the law of the conditioned stable
subordinator given its death time $\zeta$ when it starts at zero.
Using the preceding expression of the density of $\zeta$ and writing
down the finite-dimensional distributions, we see that given $\zeta=t$
the stable subordinator started at $0$ and conditioned to die at $b$
has law $\mathbb{P}_{0,b}^{\alpha,t}$.
\item[3.] The computation of the law of $Y$
given its death time, equal to $L( b/g)^\alpha$. This is
accomplished by use of the backward strong Markov property: the law of
$X_{[0,L)}$ given $g=X_{L-}$ and $\mathscr{F}^{L}$ under $\mathbb
{P}_0^\alpha$ is
$\mathbb{P}^{\alpha,L}_{0,g}$. Note that $Y$ is\vspace*{1.5pt} obtained from $X$ (on
$[0,L-)$)
by applying\vspace*{-1pt} the scaling operator $S^{\alpha}_{( b/g)^\alpha}$. By
self-similarity, the law of $Y_{[0,L( b/g)^\alpha)}$ given $g$ and
$\mathscr{F}^L$ is $\mathbb{P}^{\alpha,L( b/g)^\alpha}_{0,b}$,
which only depends
on $L( b/g)^\alpha$. It follows that the law\vspace*{-1pt} of
$Y_{[0,L( g/b)^\alpha)}$ given that its death\vspace*{1pt} time is $t$ is
$\mathbb{P}_{0,b}^{0,t}$.
\item[4.] The density of $( L,g)$. This will be
performed using the Poisson process description of the stable
subordinator and will be postponed. We prove that the law of $L$ given
$g=x$ (where $x<b$) is the law of the death time of the stable
subordinator started at zero and conditioned to die at $x$ (which does
not depend on~$b$); then $L( b/g)^\alpha$ has the law of the
death time of the stable subordinator started at zero and conditioned
to die at $b$.
\end{longlist}
Summarizing, the law of the absorption time of $Y$ is equal to the law
of the absorption time of the stable subordinator conditioned to die at
$b$ started at zero, and, conditionally on the absorption times, $Y$
and the conditioned subordinator are bridges of the stable subordinator
which start at $0$, end at $b$, and whose length is the corresponding
absorption time. We conclude that $Y$ is a stable subordinator, of
index $\alpha$, conditioned to die at $b$ and started at zero.

It remains to prove that if $L=\sup\{ s\geq0\dvtx X_s<b\}$ and $g=X_{L-}$,
then under the law of stable subordinator of index $\alpha$ started at
zero, $\mathbb{P}_0^\alpha$, the conditional law of $L$ given $g$ has density
$s\mapsto f^\alpha_s( x)/u_\alpha( x)$, where $f^\alpha
_s$ is
the density of $X_s$ under $\mathbb{P}_0^\alpha$ and $u_\alpha$ is the
potential density associated to $\mathbb{P}^\alpha_x,x\geq0$. Thanks
to the
L\'evy--It\^o decomposition of L\'evy processes (see \cite{bertoinLevyProcesses}, I.1, Theorem 1), a stable subordinator increases only by
jumps, so that under $\mathbb{P}^\alpha_0$, $X_t=\sum_{s\leq
t}\Delta X_{s}$.
(In the preceding sum, there is at most a countable quantity of nonzero
terms.) Note that if $f\dvtx\mathbb{R}_+\times[0,1]\to\mathbb{R}_+$ is
measurable,
then
\[
f( L,g)=\sum_{s}f( s,X_{s-})\mathbf{1}_{X_{s-}<b<X_{s-}+\Delta X_{s}}
\]
since only one term is positive. Since under $\mathbb{P}^\alpha_0$
the jump
process of $X$, given by $( \Delta X_{t})_{t\geq0}$ is a Poisson
point process whose characteristic measure $\pi^\alpha$ is absolutely
continuous with respect to Lebesgue measure with a density given by
\[
x\mapsto\frac{\alpha C}{\Gamma( 1-\alpha)x^{1+\alpha}}\mathbf{1}_{x>0},
\]
we can use the additive formula (cf. \cite{revuzYor}, XII.1.10, page 475) to compute
\begin{eqnarray*}
\mathbb{E}^\alpha_0( f( L,g))
&=&\mathbb{E}^\alpha_0\biggl( \sum_{s}f( s,X_{s-})\mathbf{1}_{X_{s-}<b<X_{s-}+\Delta X_{s}}\biggr)\\
&=&\int_0^\infty\mathbb{E}^\alpha_0\bigl( f( s,X_{s-})\mathbf{1}_{X_{s-}<b}\pi^\alpha\bigl( [b-X_{s-},\infty)\bigr)\bigr) \,ds\\
&=&\int_0^\infty\mathbb{E}^\alpha_0\biggl( f( s,X_{s-})\mathbf{1}_{X_{s-}<b}\frac{C}{\Gamma( 1-\alpha)( b-X_{s-})^\alpha }\biggr)\,ds.
\end{eqnarray*}
We can substitute $X_{s-}$ with $X_s$ in the preceding computation, since
\[
\mathbb{P}^\alpha_0( X_{s-}=X_s)=1,
\]
to obtain
\[
\mathbb{E}^\alpha_0( f( L_,g))=\int_0^\infty\int_0^1f( s,x)f^\alpha_s( x)\frac{C}{\Gamma( 1-\alpha)( 1-x)^\alpha}\,dx\,ds.
\]
We therefore see that the joint law of $( L,g)$ under $\mathbb
{P}^\alpha
_0$ is absolutely continuous with respect to Lebesgue measure, with a
version of the density given by
\[
( s,x)\mapsto f^\alpha_s( x)\frac{C}{\Gamma ( 1-\alpha)(1-x)^\alpha}\mathbf{1}_{0<x<1}.
\]
Using the explicit value of the potential density $u_\alpha$, we see
that the law of $g$ under $\mathbb{P}_0^\alpha$ has the density
\[
\frac{1}{\Gamma( 1-\alpha)\Gamma( \alpha)x^{1-\alpha}( 1-x)^{\alpha}},
\]
so that $g$ has the generalized arc-sine law with parameter $\alpha$.
We see then that the conditional density of $L$ given $g=x$ can be
taken equal to
\[
s\mapsto f^\alpha_s( x)C\Gamma( 1-\alpha)x^{1-\alpha}=\frac{f^\alpha_s( x)}{u_\alpha( x)}
\]
as announced.

\begin{appendix}\label{app}

\section*{Appendix: Finitude and bicontinuity of\break the resolvent density}
We now prove Proposition \ref{finitudeOfResolventDensity}. Recall that
$f_t$ denotes the (bounded and positive) density of the L\'evy process
at time $t$; since the process has jumps of both signs, it is known
that $f_1( x)/|x|^{1+\alpha}$ converges as $x\to\pm\infty
$ to
positive constants $c_{\pm}$. We also have the scaling
identity
\[
f_t( x)=f_1( x t^{-1/\alpha})t^{-1/\alpha},
\]
from which we can deduce the asymptotic behavior of $f_t( x)$ as
$x\to\infty$ and $t\to0$:
%
\begin{equation}\label{asymptoticsStableDensity}
\mbox{if }t^{-1/\alpha}x\to\pm\infty\qquad\mbox{then }f_t( x)\sim\frac{c_{\pm}t}{x^{1+\alpha}}.
\end{equation}

We begin by analyzing the finitude of the resolvent density
\[
v_\lambda( x,y)=\int_0^\infty q_t( x,y)e^{-\lambda t} \,dt.
\]
From equations (\ref{transitionDensityOU}) and
(\ref{asymptoticsStableDensity}),
\[
\lim_{t\to\infty}q_t( x,y)=f_1( y),
\]
as expected since $Y$ has a stationary distribution with density $f_1$,
and so
\[
\int_a^\infty q_t( x,y)e^{-\lambda t} \,dt<\infty
\]
for all $a>0$. On the other hand, if $x,y=0$ then
\[
q_t( 0,0)=p_{e^t-1}( 0,0)=f_{e^{t}-1}( 0)=f_1( 0)\frac{1}{(e^t-1)^{1/\alpha}}\sim f_1( 0)\frac{1}{t^{1/\alpha}}
\]
so that
\[
\int_0^a q_t( 0,0)e^{-\lambda t}\, dt<\infty\quad\mbox{if and only if}\quad\alpha\in(1,2).
\]
If $x\neq y$, then
\[
|( ye^{t/\alpha}-x)( e^t-1)^{-1/\alpha}|\to\infty\qquad\mbox{as }t\to0+,
\]
so that
\begin{eqnarray*}
q_t( x,y)
&=&f_1\bigl( ( ye^{t/\alpha}-x)( e^t-1)^{-1/\alpha }\bigr)e^{t/\alpha}(e^t-1)^{-1/\alpha}\\
&\sim&\frac{c_{\pm}}{|y-x|^{1+\alpha}}( e^t-1)\to 0\qquad\mbox{as $t\to0+$.}
\end{eqnarray*}
Hence,
\[
\int_0^a q_t( x,y)e^{-\lambda t}\, dt<\infty\qquad\mbox{if }x\neq y.
\]
The only remaining case is $x=y\neq0$, and we have
\[
q_t( x,x)=f_1\bigl( x( e^{t/\alpha}-1)( e^t-1)^{-1/\alpha}\bigr)e^{t/\alpha}(
e^t-1)^{-1/\alpha}.
\]
Since
\[
\lim_{t\to
0+}( e^{t/\alpha}-1)( e^t-1)^{-1/\alpha}=
\cases{
0,&\quad if $\alpha>1$,\cr
1,&\quad if $\alpha=1$,\cr
\infty,&\quad if $\alpha<1$,}
\]
then
\[
q_t( x,x)\sim
\cases{
\displaystyle f_1( 0)( e^t-1)^{-1/\alpha},&\quad if $\alpha\in(1,2)$,\cr
\displaystyle f_1( x)( e^t-1)^{-1},&\quad if $\alpha=1$,\vspace*{2pt}\cr
\displaystyle \frac{c_{\pm}}{|x|^{\alpha+1}}t^{-\alpha}\alpha^{1+\alpha},&\quad if $\alpha\in(0,1)$,}
\qquad\mbox{as $t\to0+$}.
\]
We conclude that if $x\neq0$,
\[
v_\lambda( x,x)<\infty\quad\Leftrightarrow\quad\alpha\neq1.
\]

We proceed by analyzing the continuity of $v_\lambda$. We argued, using
the lower semicontinuity of $v_\lambda$, that it is continuous when it
is infinite. It therefore remains to see if it is continuous where it
is finite. For $\alpha\in(1,2)$, we should show that $v_\lambda$ is
continuous everywhere, an assertion which is easily handled: since
$f_1$ is bounded, say by $M$, then
%
\begin{equation}\label{oUDensityBound}
q_t( x,y)\leq Me^{t/\alpha}( e^t-1)^{-1/\alpha};
\end{equation}
the right-hand side multiplied by $e^{-\lambda t}$ is integrable on
$[0,\infty)$ for every $\lambda>0$ and so, by dominated convergence,
$v_\lambda$ is continuous and bounded when $\alpha\in(1,2)$.
Actually, for any $\alpha\in(0,2)$, $\lambda>0$ and $a>0$,
\[
( x,y)\mapsto\int_a^\infty q_t( x,y) e^{-\lambda t}\, dt
\]
is continuous and bounded. This happens since $\sup_{t\geq\varepsilon
,\
x,y\in\mathbb{R}}q_t( x,y)<\infty$ by (\ref{oUDensityBound}). It is
therefore sufficient to study the behavior of
\[
( x,y)\mapsto\int_0^a q_t( x,y) e^{-\lambda t}\, dt
\]
for some (judiciously chosen) $a>0$. To do this, note that since $f_1$
is bounded and because of its asymptotic behavior recalled in (\ref{asymptoticsStableDensity}), there exist constants $D,M>0$ such that for
any $b>0$:
\[
f_1( x)\leq
\cases{
\displaystyle M,&\quad$|x|\leq b$,\vspace*{2pt}\cr
\displaystyle\frac{D}{|x|^{1+\alpha}},&\quad$|x|>b$.}
\]
By scaling, it follows that
\[
f_t( x)\leq
\cases{
\displaystyle Mt^{-1/\alpha},&\quad$\displaystyle|x|\leq bt^{1/\alpha}$,\cr
\displaystyle \frac{D t}{|x|^{1+\alpha}},&\quad$\displaystyle|x|>bt^{1/\alpha}$.}
\]
Hence,
%
\begin{equation}\label{boundOUDensity}
q_t( x,y)\leq
\cases{
\displaystyle M( e^t-1)^{-1/\alpha},&\quad$\displaystyle|ye^{t/\alpha}-x|/( e^t-1)^{1/\alpha}\leq
b$,\cr
\displaystyle D\frac{e^t-1}{|ye^{t/\alpha}-x|^{1+\alpha}},&\quad$\displaystyle |ye^{t/\alpha }-x|/(
e^t-1)^{1/\alpha}> b$.
}
\end{equation}
To study the continuity of $v_\lambda$, we will make a careful
analysis implementing (\ref{boundOUDensity}). Let us first show that
$v_\lambda$ is bicontinuous at $( x,y)$ if $x\neq y$. First,
consider $\varepsilon,a>0$ such that
%
\begin{equation}\label{boundOnInfimum}
\mathop{\inf_{|x'-x|,|y'-y|\leq
\varepsilon}}_{t\leq a}|y'e^{t/\alpha}-x'|=\rho>0,
\end{equation}
and then use the second bound of (\ref{boundOUDensity}), with $b$
small enough, together with (\ref{boundOnInfimum}) to obtain
\[
\mathop{\sup_{|x'-x|,|y'-y|\leq\varepsilon}}_{t\leq a}q_t( x',y')\leq D\frac{e^a-1}{\rho^{1+\alpha}}.
\]
From the above, we conclude the continuity of $( x',y')\mapsto
\int_0^a q_t( x',y')e^{-\lambda t}\, dt$ at $( x,y)$.

It remains to verify the continuity of $v_\lambda$ at $( y,y)$ if
$\alpha\in(0,1)$ and $y\neq0$; for concreteness we will assume that
$y>0$. We have to argue separately that
\[
\mathop{\lim_{x,z\to y}}_{x\leq
z}v_\lambda( x,z)=v_\lambda( y,y)\quad\mbox{and}\quad\mathop{\lim_{x,z\to y}}_{x>
z}v_\lambda( x,z)=v_\lambda( y,y).
\]

$x\leq z$. Choose $\varepsilon>0$ such that $y-\varepsilon>0$;
if $x\leq z$ and
$z\geq y-\varepsilon$ then
\[
|ze^{t/\alpha}-x|\geq
z( e^{t/\alpha}-1)\geq( y-\varepsilon)( e^{t/\alpha}-1).
\]
Since $\alpha\in(0,1)$, then
\[
\lim_{t\to0}\frac{e^{t/\alpha}-1}{( e^t-1)^{1/\alpha}}\to0
\]
and so there exists $a>0$ such that
\[
\mathop{\inf_{y-\varepsilon\leq z,x\leq z}}_{0\leq t\leq a}\frac{|ze^{t/\alpha}-x|}{( e^t-1)^{1/\alpha}}\geq D.
\]
We can then continue from (\ref{boundOnInfimum}).

$x>z$. Here is where we have to be most careful since
\[
\mathop{\inf_{x,z\in B_{\varepsilon}( y)}}_{t\leq a}\frac{|ze^{t/\alpha}-x|}{( e^t-1)^{1/\alpha}}=0
\]
for all $\varepsilon,a>0$ and so the bounds used previously no longer work.

For $x>z$, let us introduce the function
\[
t\mapsto\frac{ze^{t/\alpha}-x}{( e^t-1)^{1/\alpha}},
\]
which tends to $-\infty$ as $t$ decreases to zero, tends to $z$ when
$t$ goes to infinity, touches $0$ at $\alpha\log( x/z)$, and
since its derivative is given by
\[
\frac{xe^t-ze^{t/\alpha}}{\alpha( e^t-1)^{1+1/\alpha}},
\]
it is increasing on
\[
(0,\alpha\log( x/z)/( 1-\alpha)]
\]
and decreasing on
\[
[\alpha\log( x/z)/( 1-\alpha),\infty).
\]
The function
\[
\phi\dvtx t\mapsto\frac{|ze^{t/\alpha}-x|}{( e^t-1)^{1/\alpha}}
\]
will be important in what follows because it governs, by means of
(\ref{boundOUDensity}), the choice of the bound on $q_t( x,z)$.

If $b\leq z$ then $\phi$ equals $b$ at two points, say $t_1$ and $t_2$,
delimiting the three regions on which we can bound $q_t$:
%
\begin{equation} \label{boundOnIntegrand}
q_t( x,z)e^{-\lambda t}\leq
\cases{
\displaystyle\frac{D( e^t-1)}{( x-ze^{t/\alpha})^{1+\alpha}},&\quad\mbox
{if $t\leq t_1$},\cr
\displaystyle M( e^t-1)^{-1/\alpha},&\quad if $t\in[t_1,t_2]$,\cr
\displaystyle\frac{D( e^t-1)}{( ze^{t/\alpha}-x)^{1+\alpha}},&\quad if $t\geq t_2$.}
\end{equation}
There is an obvious problem with the second region since there the
upper bound is asymptotic to $Mt^{-1/\alpha}$ which is not integrable
on $(0,\varepsilon)$ for any positive $\varepsilon$ since $\alpha\in(0,1)$.

Let us start with the first region: we had assumed that $y>0$ and so
$z>0$ if it is close enough to $y$. Let $d>0$ and set
\[
r=\alpha\log\bigl( x/z\bigl( 1-d( x-z)^{1/\alpha}\bigr)\bigr).
\]
Note that
\[
\mathop{\lim_{x,z\to y}}_{ x>z}\phi( r)=\frac{d y^{1+1/\alpha}}{\alpha^{1/\alpha}}
\]
so that $t_1\leq r$ when $d$ is small enough. We would like to see that
\[
\mathop{\limsup_{x,z\to y}}_{x>z}\int_0^r\frac{e^t-1}{( x-ze^{t/\alpha})^{1+\alpha}}\,dt=0
\]
or equivalently
\begin{equation}\label{integral}
\mathop{\limsup_{x,z\to y}}_{x>z}\int_0^r\frac{t}{( x-ze^{t/\alpha})^{1+\alpha}}\,dt=0.
\end{equation}
Since
\[
\frac{d}{dt}\frac{t}{( x-ze^{t/\alpha})^{1+\alpha}}=\frac
{1}{\alpha( x-ze^{t/\alpha})^{2+\alpha}}\bigl( \alpha x-\alpha
ze^{t/\alpha}+tze^{t/\alpha}( 1+\alpha)\bigr),
\]
we see that the integrand in (\ref{integral}), denoted $\psi$, is
increasing on $[0,r]$, going from $0$ to
\[
\frac{r}{x^{1+\alpha} d^{1+\alpha}( x-z)^{1+1/\alpha}}
\leq\frac{\alpha}{z x^{1+\alpha}d( x-z)^{1/\alpha}},
\]
where the upper bound follows from $\log( 1+t)\leq t$. Note that
if $r_0=0$ and
\[
r_n=\alpha\log\bigl( x/z\bigl(1-d_n( x-z)^{1/(1+\alpha)}\bigr)\bigr),
\]
where $( d_n)$ decreases to zero then
\[
\psi( r_n)\leq\frac{\alpha}{z x^{1+\alpha}d_n^{1+\alpha
}},\qquad
\mathop{\lim_{x,z\to
y}}_{x>z}\psi( r_n)=\frac{\alpha}{y^{2+\alpha} d_n^{1+\alpha}}
\]
and
\[
r_n-r_{n-1}\leq\alpha( x-z)^{1/(1+\alpha)}( d_{n-1}-d_n)\frac{1}{( 1-d_{n-1}(x-z)^{1/(1+\alpha)})}
\]
so that
\begin{eqnarray*}
&&\int_{r_{n-1}}^{r_n}\frac{t}{( x-ze^{t/\alpha})^{1+\alpha
}} \,dt\\
&&\qquad\leq\frac{\alpha^2}{z x^{1+\alpha}}( x-z)^{1/(1+\alpha)}\frac{d_{n-1}-d_n}{d_n^{1+\alpha}}\frac{1}{( 1-d_{n-1}(x-z)^{1/(1+\alpha)})}.
\end{eqnarray*}
If $N$ is such that $r\leq r_N$, then
\[
\mathop{\limsup_{x,z\to y}}_{ x>z}\int_0^r\psi( t) \,dt\leq\frac{\alpha^2}{y^{2+\alpha}} \mathop{\limsup_{x,z\to y}}_{x>z}( x-z)^{1/(1+\alpha)}\sum_{1\leq n\leq N-1}\frac{d_{n}-d_{n+1}}{d_{n+1}^{1+\alpha}}.
\]
Let $d_n=\delta/\sqrt{n}$, which is so chosen so that
\[
\psi( r_n)\leq
\frac{\alpha}{z}\biggl( \frac{\sqrt{n}}{x\delta}\biggr)^{1+\alpha};
\]
it also implies that
\[
\frac{d_{n}-d_{n+1}}{d_{n}^{1+\alpha}}\sim\frac{
\delta^{-\alpha}}{2n^{1-\alpha/2}}\qquad\mbox{as }n\to\infty.
\]
Finally, note that if $N$ is bigger than $\delta/d( x-z)^{-1/\alpha(
1+\alpha)}$ but taken asymptotic to it as
$x,z\to y$, then $r\leq r_N$ and
\[
\mathop{\limsup_{x,z\to y}}_{x>z}\int_0^r\psi( t)\, dt\leq
\frac{\alpha^2\delta^{-\alpha}}{y^{2+\alpha}\alpha}\mathop{\limsup_{x,z\to y}}_{x>z}( x-z)^{1/(1+\alpha)}N^{\alpha/2}/2.
\]
Since $N^{\alpha}\sim( \delta/d)^{\alpha/2}( x-z)^{1/(1+\alpha)}$ then
\[
\mathop{\limsup_{x,z\to y}}_{x>z}\int_0^r\psi( t) \,dt\leq\frac{\alpha^2\delta^{-\alpha/2}}{2y^{2+\alpha}d^{\alpha/2}\alpha},
\]
which can be made as small as we want by taking $\delta$ big enough.

We now consider the second region showing that for any
$y>0$
\[
\limsup_{b\to0}\mathop{\limsup_{x,z\to y}}_{x>z}\int_{t_1}^{t_2}( e^t-1)^{-1/\alpha}\, dt=0
\]
or equivalently
\[
\mathop{\lim_{x,z\to y}}_{x>z}\int_{t_1}^{t_2}t^{-1/\alpha}\, dt\leq\frac{2\alpha b}{y}.
\]
This will accomplished by means of a lower bound on $t_1$ and an upper
bound for $t_2$, valid as $x,z\downarrow y$; the bounds on $t_1$ and
$t_2$ are obtained as in the analysis of the first region: recall that if
\[
r_{\pm}=\alpha\log\biggl( \frac{x}{z}\bigl( 1\pm d( x-z)^{1/\alpha}\bigr)\biggr)\qquad\mbox{then } \mathop{\lim_{x,z\to y}}_{x>z}\phi( r_{\pm})
=\frac{dy^{1/\alpha+1}}{\alpha^{1/\alpha}}.
\]
Hence, for arbitrary $d> b\alpha^{1/\alpha}/y^{1/\alpha+1}$
\[
r_{-}\leq t_1\leq t_2\leq r_{+}
\]
for $x,z$ close enough to $y$. We then obtain
\[
\mathop{\lim_{x,z\to y}}_{x>z}\int_{t_1}^{t_2}t^{-1/\alpha}\, dt
\leq
\mathop{\lim_{x,z\to y}}_{x>z}\frac{r_+-r_-}{r_-^{1/\alpha}}
= \frac{2dy^{1/\alpha}}{\alpha^{1/\alpha-1}},
\]
so that
\[
\limsup_{x\to y+}\int_{t_1}^{t_2}t^{-1/\alpha} \,dt\leq\frac{2\alpha b}{y}.
\]

On the third region,
\[
\mathop{\lim_{x,z\to y}}_{x>z}\int_{t_2}^{\infty}q_t( x,z)e^{-\lambda t}\,dt=v_\lambda( y,y),
\]
as we now see. This implies the bicontinuity of $v_\lambda$ at $(
y,y)$. It suffices to prove that for any $a>0$,
\[
\mathop{\lim_{x,z\to y}}_{x>z}\int_{t_2}^{a}q_t( x,z)e^{-\lambda t}\, dt=\int_{0}^{a}q_t( y,y)e^{-\lambda t}\, dt,
\]
which we achieve by arguing as for the first region. First, note that
on $[t_2,a]$, for small enough $a$, we have
\[
q_t( x,z)\leq\frac{D( e^t-1)}{( ze^{t/\alpha}-x)^{1+\alpha}}\leq\frac{2Dt}{( ze^{t/\alpha}-x)^{1+\alpha}}.
\]
The rightmost bound, denoted $\psi$, is decreasing on $[t_2,a]$ and by setting
\[
r_1=\alpha\log\bigl( x/z\bigl( 1-d_1( x-z)^{1/(1+\alpha)}\bigr)\bigr),
\]
we see that $\psi( t)$ is uniformly bounded on $[r_1,a]$ as
$x,z\to y$, so that by dominated convergence,
\[
\mathop{\lim_{x,z\to y}}_{x>z}\int_{r_1}^{a}q_t( x,z)e^{-\lambda t}\, dt=\int_{0}^{a}q_t( y,y)e^{-\lambda t}\, dt.
\]
It remains to see that
\[
\mathop{\lim_{x,z\to y}}_{x>z}\int_{t_2}^{r_1}\frac{t}{( ze^{t/\alpha}-x)^{1+\alpha}}\, dt=0;
\]
for this we adapt the analysis of the first region from equation (\ref{integral}), using
\[
r_n=\alpha\log\bigl( x/z\bigl( 1+d_n( x-z)^{1/(1+\alpha)}\bigr)\bigr),
\]
where $d_n$ decreases to zero.
\end{appendix}

\section*{Acknowledgments}
The authors would like to thank Marc Yor for his comments regarding a
preliminary version of this work and for his ``Gaussian'' proof of
Theorem \ref{pathwiseTheorem} in the Brownian case. They also thank Jim
Pitman for valuable discussions on the relevance of the Feller property
for the construction of bridges which arrive continuously at their
endpoint. GUB would like to thank his Ph.D. supervisors, J. Bertoin and
M. E. Caballero, for their help, support and encouragement during the
development of this work. The authors would also like
to thank the anonymous referee, as well as Guillaume Coqueret, for
their
helpful comments and unearthing of inaccuracies.

%

\printaddresses

\end{document}